\documentclass[eqthmnum]{jt-calcs}
\usepackage{mathrsfs}
\usepackage{booktabs}
\usepackage{longtable}
\usepackage{multirow}
\usepackage{tikz-cd}
\usepackage{makeidx}
\usepackage[columns=5]{idxlayout}

\usepackage[backend=bibtex,style=alphabetic]{biblatex}
\bibliography{bibliography}

\newlist{properties}{enumerate}{1}
\setlist[properties]{label=(P\arabic*),ref=P\arabic*}

\crefname{propertiesi}{property}{properties}
\Crefname{propertiesi}{Property}{Properties}
\crefformat{propertiesi}{property (#2#1#3)}
\crefrangeformat{propertiesi}{properties (#3#1#4) to (#5#2#6)}
\crefmultiformat{propertiesi}{properties (#2#1#3)}{ and (#2#1#3)}{, (#2#1#3)}{ and (#2#1#3)}

\newcommand{\Clu}{\ensuremath{\mathfrak{Cl}_{\mathrm{U}}}}

\crefname{enumi}{}{}
\crefformat{enumi}{#2#1#3}

\newcommand{\Ql}{\ensuremath{\overline{\mathbb{Q}}_{\ell}}}

\renewcommand{\epsilon}{\varepsilon}

\newenvironment{assumption}{
\begin{center}

\begin{tabular}{!{\vrule width 0.75pt}p{15cm}!{\vrule width 0.75pt}}
\noalign{\hrule height 0.75pt}
\cellcolor[gray]{0.85}
\ignorespaces
}
{
\\[.5ex]
\noalign{\hrule height 0.75pt}
\end{tabular}
\end{center}
\ignorespacesafterend
}

\renewcommand{\cref}{\Cref}

\title{Finding Characters Satisfying a Maximal Condition for Their Unipotent Support}
\author{Jay Taylor}
\address{FB Mathematik, TU Kaiserslautern, Postfach 3049, 67653 Kaiserslautern, Germany.}
\email{taylor@mathematik.uni-kl.de}

\makeindex

\begin{document}
\start
\begin{abstract}
In this article we extend independent results of Lusztig and H\'{e}zard concerning the existence of irreducible characters of finite reductive groups, (defined in good characteristic and arising from simple algebraic groups), satisfying a strong numerical relationship with their unipotent support. Along the way we obtain some results concerning quasi-isolated semisimple elements.
\end{abstract}

\section{Introduction}
Throughout this article $\bG$ will be a simple algebraic group over $\mathbb{K} = \overline{\mathbb{F}}_p$\index{K@$\mathbb{K}$}, an algebraic closure of the finite field $\mathbb{F}_p$, where $p$ is a good prime for $\bG$. Furthermore, we assume $\bG$ is defined over $\mathbb{F}_q \subset \overline{\mathbb{F}}_p$ (where $q$ is a power of a $p$) and $F:\bG\to \bG$\index{F@$F$} is the associated Frobenius endomorphism. Throughout we will denote an algebraic group in bold and its corresponding fixed point subgroup in roman, for instance $G := \bG^F$.

Let us denote by $\Clu(\bG)$\index{CluG@$\Clu(\bG)$} the set of all unipotent conjugacy classes of $\bG$ and by $\Clu(\bG)^F$\index{CluGF@$\Clu(\bG)^F$} all those classes which are $F$-stable. In \cite[\S13.4]{lusztig:1984:characters-of-reductive-groups} and \cite[\S10]{lusztig:1992:a-unipotent-support} Lusztig defined combinatorially using $j$-induction and the Springer correspondence a map $\Phi_G : \Irr(G) \to \Clu(\bG)^F$. He had previously conjectured in \cite{lusztig:1980:some-problems-in-the-representation-theory} that for any $\bG$ and any irreducible character $\chi \in \Irr(G)$ there should exist a unique class $\mathcal{O}_{\chi} \in \Clu(\bG)^F$ of maximal dimension satisfying
\begin{equation*}
\sum_{g \in \mathcal{O}_{\chi}^F} \chi(g) \neq 0.
\end{equation*}
It was shown in \cite{lusztig:1992:a-unipotent-support}, (and later \cite{geck:1996:on-the-average-values}), that in good characteristic $\mathcal{O}_{\chi}$ always exists and that $\Phi_G(\chi) = \mathcal{O}_{\chi}$. We call the class $\mathcal{O}_{\chi}$ the \emph{unipotent support} of $\chi$.

Recall that Lusztig has shown for each $\chi \in \Irr(G)$ there is a well defined integer $n_{\chi}$\index{nchi@$n_{\chi}$} such that $n_{\chi}\cdot \chi(1)$ is a polynomial in $q$ with integer coefficients. If $x \in \bG$ then we write $A_{\bG}(x)$\index{AGx@$A_{\bG}(x)$} for the component group $C_{\bG}(x)/C_{\bG}(x)^{\circ}$, furthermore if $F(x) = x$ then we denote again by $F$ the automorphism of $A_{\bG}(x)$ induced by $F$. We will denote by $A_G(x)$\index{AGx@$A_G(x)$} the quotient group $C_{\bG}(x)^F/{C_{\bG}(x)^{\circ}}^F$ which may be identified with the fixed point group $A_{\bG}(x)^F$. In \cite[\S13.4]{lusztig:1984:characters-of-reductive-groups} Lusztig gave various properties of the map $\Phi_G$ which should hold when $Z(\bG)$ is connected. For our concerns the most important of these are that: $\Phi_G$ is surjective, $n_{\chi}$ divides $|A_{\bG}(u)|$ for $u \in \mathcal{O}_{\chi}$ and finally that there exists for each $\mathcal{O} \in \Clu(\bG)^F$ at least one $\chi \in \Phi_G^{-1}(\mathcal{O})$ such that $n_{\chi} = |A_{\bG}(u)|$ for $u \in \mathcal{O}$.

This intriguing numerical property has provided several interesting applications to the representation theory of $G$, namely via Kawanaka's theory of generalised Gelfand--Graev representations, (see for instance \cite[\S3]{geck:1999:character-sheaves-and-GGGRs} and \cite[Theorem 4.5]{geck-hezard:2008:unipotent-support}). Unfortunately, at the time of writing, proving the existence of such characters seems only possible by carrying out a case by case check and the details of this were omitted from \cite[\S13.4]{lusztig:1984:characters-of-reductive-groups}. A detailed case by case analysis was provided by H\'{e}zard in his PhD thesis \cite{hezard:2004:thesis} and also partly by Lusztig in \cite{lusztig:2009:unipotent-classes-and-special-Weyl}. Note that in the latter reference necessary questions concerning $F$-stability were not addressed, however results concerning groups with a disconnected centre were considered.

It is the main goal of this paper, (using extensively the work of Lusztig and H\'{e}zard), to extend in a natural way the existence of characters satisfying $n_{\chi} = |A_{\bG}(u)|$ for their unipotent support to the case where $\bG$ has a disconnected centre. Along the way we also restate the results of H\'{e}zard so that all simple groups are treated in our paper. Note that the existence of the characters mentioned above will follow from our main result, (\cref{prop:A}), as is shown in \cite[Theorem 3.1]{taylor:2011:on-unipotent-supports}. In \cite{taylor:2011:on-unipotent-supports} it is also shown that this result gives direct applications to the representation theory of those groups $G$ along the same lines as those obtained by Geck and Geck--H\'{e}zard.

The logical layout of this paper is as follows. In \cref{sec:state-result} we introduce enough notation so that we can accurately state our main result. In \cref{sec:the-setup,sec:classes-and-chars} we introduce further notation and conventions that will be required for the case by case check. \Cref{sec:clarification-typeD} is dedicated to dealing with clarifications required concerning degenerate elements in half-spin groups. In particular we describe explicitly the Springer correspondence for degenerate elements, which may be of independent interest. In \cref{sec:quasi-isolated} we recall results of Bonnaf\'{e} concerning quasi-isolated semisimple elements and prove the existence of $F$-stable classes. Finally the remaining sections are the execution of the case by case check. Note that we have also included an index of notation for the convenience of the reader.

\begin{acknowledgments}
The work presented in this article forms the main part of the authors PhD qualification, which was supervised by Prof.\ Meinolf Geck. The author wishes to express his deepest thanks to Prof.\ Geck for proposing this problem and for his vital guidance. The author would like to thank Prof.\ C\'{e}dric Bonnaf\'{e} for many useful discussions and Professors Gunter Malle and Radha Kessar for carefully reading his PhD thesis and providing numerous useful comments. We also thank Prof.\ Malle for his comments on a preliminary version of this paper.
\end{acknowledgments}

\section{Stating The Result}\label{sec:state-result}
\begin{pa}
Let us fix a regular embedding $\iota : \bG \hookrightarrow \btG$\index{iota@$\iota$}\index{Gtilde@$\btG$} of $\bG$\index{G@$\bG$} into a group with connected centre with the same derived subgroup as $\bG$, (see \cite[\S7]{lusztig:1988:reductive-groups-with-a-disconnected-centre}), and denote by $\iota^{\star} : \btG^{\star} \twoheadrightarrow \bG^{\star}$\index{iotastar@$\iota^{\star}$}\index{Gtildestar@$\btG^{\star}$}\index{Gstar@$\bG^{\star}$} the induced surjective morphism of dual groups. If $\bG$ has a connected centre then we simply take $\btG = \bG$ and $\iota$ to be the identity map. We assume these groups to be fixed and $F^{\star}$\index{Fstar@$F^{\star}$} to be a corresponding Frobenius endomorphism of the dual groups $\bG^{\star}$ and $\btG^{\star}$. Assume $\bT_0 \leqslant \bB_0$\index{T00@$\bT_0$}\index{B_0@$\bB_0$} is a maximal torus and Borel subgroup of $\bG$, both assumed to be $F$-stable, then then we denote by $(\bW,\mathbb{S})$\index{W@$\bW$}\index{S@$\mathbb{S}$} the Coxeter system of $\bG$ defined with respect to $\bT_0\leqslant \bB_0$. Taking $\btT_0 \leqslant \btB_0$\index{T0tilde@$\btT_0$}\index{B0tilde@$\tilde{\bB}_0$} to be the unique maximal torus and Borel subgroup of $\tilde{\bG}$ satisfying $\bT_0 = \btT_0 \cap \bG$ and $\bB_0 = \tilde{\bB}_0 \cap \bG$ we have $\iota$ naturally identifies $(\bW,\mathbb{S})$ with the Coxeter system of $\tilde{\bG}$ defined with respect to $\btT_0 \leqslant \btB_0$. We now fix $F^{\star}$-stable maximal tori $\btT_0^{\star} \leqslant \btG^{\star}$\index{T0tildestar@$\btT_0^{\star}$} and $\bT_0^{\star} \leqslant \bG^{\star}$\index{T0star@$\bT_0^{\star}$} such that the triples $(\btG,\btT_0,F)$ and $(\btG^{\star},\btT_0^{\star},F^{\star})$, respectively $(\bG,\bT_0,F)$ and $(\bG^{\star},\bT_0^{\star},F^{\star})$, are in duality. We then necessarily have that $\iota^{\star}(\btT_0^{\star}) = \bT_0^{\star}$, (see for instance \cite[Lemma 1.71]{taylor:2012:thesis}). Denote by $(\bW^{\star},\mathbb{T})$\index{Wstar@$\bW^{\star}$}\index{T@$\mathbb{T}$} the Coxeter system of $\btG^{\star}$ defined with respect to $\btT_0^{\star} \leqslant \btB_0^{\star}$ then we have $\iota^{\star}$ naturally identifies $(\bW^{\star},\mathbb{T})$ with the Coxeter system of $\bG^{\star}$ defined with respect to $\bT_0^{\star}\leqslant \bB_0^{\star} := \iota^{\star}(\btB_0^{\star})$\index{B0star@$\bB_0^{\star}$}. Recall that by duality we have an anti-isomorphism $\bW \to \bW^{\star}$ denoted by $w \mapsto w^{\star}$\index{wstar@$w^{\star}$} which satisfies $\mathbb{S}^{\star} = \mathbb{T}$.
\end{pa}

\begin{pa}
For each $s \in \bT_0^{\star}$ we denote by $\bW^{\star}(s)^{\circ} \leqslant \bW^{\star}(s) \leqslant \bW^{\star}$\index{Wstarscirc@$\bW^{\star}(s)^{\circ}$}\index{Wstars@$\bW^{\star}(s)$} the Weyl group of the connected reductive group $C_{\bG^{\star}}(s)^{\circ}$ and the group $\{w \in \bW^{\star} \mid s^w = \dot{w}^{-1}s\dot{w} = s\}$ of all elements commuting with $s$, (where $\dot{w} \in N_{\bG^{\star}}(\bT_0^{\star})$ is a representative of $w$). We define similar groups for all semisimple elements $\tilde{s} \in \btT_0^{\star}$ but in this case we have $\bW^{\star}(\tilde{s})^{\circ} = \bW^{\star}(\tilde{s})$\index{Wtildesstar@$\bW^{\star}(\tilde{s})$} as the centraliser of every semisimple element is connected. Assume the conjugacy class of $\tilde{s}$ is $F^{\star}$-stable then there exists $w \in \bW^{\star}$ such that $F^{\star}(\tilde{s}) = \tilde{s}^{\dot{w}}$. Choose a Borel subgroup $\bB(\tilde{s})$ of $C_{\btG^{\star}}(\tilde{s})$ containing $\btT_0^{\star}$ then by \cite[Lemma 1.9(i)]{lusztig:1984:characters-of-reductive-groups} we may assume $w$ is the unique element of minimal length in the right coset $w\bW^{\star}(\tilde{s})$; this element is characterised by the fact that $w$ normalises $\bB(\tilde{s})$. We may now define a Frobenius endomorphism $F_{\tilde{s}}^{\star}$\index{Fstildestar@$F_{\tilde{s}}^{\star}$} of $C_{\btG^{\star}}(\tilde{s})$ by setting $F_{\tilde{s}}^{\star}(h) := {}^{\dot{w}}F^{\star}(h)$ which stabilises $\btT_0^{\star}$ and $\bB(\tilde{s})$ hence induces a Coxeter automorphism of $\bW^{\star}(\tilde{s})$. Taking $\bB(s) := \iota^{\star}(\bB(\tilde{s}))$ to be a Borel subgroup of $C_{\bG^{\star}}(s)^{\circ}$ one similarly has a Frobenius endomorphism $F_s^{\star}$\index{Fsstar@$F_s^{\star}$} of $C_{\bG^{\star}}(s)$ inducing a Coxeter automorphism of $\bW^{\star}(s)^{\circ}$ and satisfying $F_s^{\star}\circ\iota^{\star} = \iota^{\star} \circ F_{\tilde{s}}^{\star}$. We will similarly denote by $\bW(\tilde{s})$\index{Wtildes@$\bW(\tilde{s})$}, $\bW(s)$\index{Ws@$\bW(s)$} and $\bW(s)^{\circ}$\index{Wscirc@$\bW(s)^{\circ}$} the corresponding subgroups of $\bW$ under the anti-isomorphism between $\bW$ and $\bW^{\star}$.
\end{pa}

\begin{pa}
By the Jordan decomposition of characters and \cite[Proposition 13.20]{digne-michel:1991:representations-of-finite-groups-of-lie-type} we have a bijection
\begin{equation}\label{eq:char-bijection}
\Psi_{\tilde{s}} : \mathcal{E}(\tilde{G},\tilde{s}) \to \mathcal{E}(C_{\tilde{G}^{\star}}(\tilde{s}),1) \to \mathcal{E}(C_{G^{\star}}(s)^{\circ},1)\index{Psitildes@$\Psi_{\tilde{s}}$}
\end{equation}
where the first set is the geometric Lusztig series of $\tilde{G}$ determined by $\tilde{s} \in \btT_0^{\star}$ and the latter sets are the sets of unipotent characters. Note that here we denote by $C_{\tilde{G}^{\star}}(\tilde{s})$, (resp.\ $C_{G^{\star}}(s)^{\circ}$), the fixed point subgroup of $C_{\btG^{\star}}(\tilde{s})$, (resp.\ $C_{\bG^{\star}}(s)^{\circ}$), under $F_{\tilde{s}}^{\star}$, (resp.\ $F_s^{\star}$). By the classification of unipotent characters given in \cite[\S4]{lusztig:1984:characters-of-reductive-groups} the set of unipotent characters $\mathcal{E}(C_{\tilde{G}^{\star}}(\tilde{s}),1)$ can further be partitioned in to what are called families; we denote these by $\mathcal{\tilde{F}}$\index{F@$\tilde{\mathcal{F}}$}. In fact, we have an injective map $\mathcal{\tilde{F}} \mapsto \bW^{\star}(\tilde{\mathcal{F}}) \subseteq \Irr(\bW^{\star}(\tilde{s}))$\index{WstarF@$\bW^{\star}(\tilde{\mathcal{F}})$} whose image is the set of $F_{\tilde{s}}^{\star}$-stable families of irreducible characters of $\bW^{\star}(\tilde{s})$. By the map in \cref{eq:char-bijection} we have a bijection $\tilde{\mathcal{F}} \mapsto \mathcal{F}$\index{F@$\mathcal{F}$} between the families of unipotent characters of $C_{\tilde{G}^{\star}}(\tilde{s})$ and those of $C_{G^{\star}}(s)^{\circ}$. With this in mind we denote by $\overline{\mathcal{T}}_{\btG}$\index{TbarGtilde@$\overline{\mathcal{T}}_{\btG}$} the set of all pairs $(\tilde{s},\bW^{\star}(\tilde{\mathcal{F}}))$ satisfying
\begin{itemize}
	\item $\tilde{s} \in \btT_0^{\star}$ lies in an $F^{\star}$-stable conjugacy class and the image of $s := \iota^{\star}(\tilde{s})$ under an adjoint quotient of $\bG^{\star}$ is a quasi-isolated semisimple element, (for a definition see \cite[\S1.B]{bonnafe:2005:quasi-isolated}),
	\item $\bW^{\star}(\widetilde{\mathcal{F}}) \subseteq \Irr(\bW^{\star}(\tilde{s}))$ is a family of characters which is invariant under the induced action of $F_{\tilde{s}}^{\star}$.
\end{itemize}
Clearly $\bW^{\star}$ acts naturally by conjugation on $\overline{\mathcal{T}}_{\btG}$ and we denote by $\mathcal{T}_{\btG}$\index{TGtilde@$\mathcal{T}_{\btG}$} the orbits under this action.
\end{pa}

\begin{pa}
Let $\mathcal{N}_{\bG}$\index{NG@$\mathcal{N}_{\bG}$} denote the set of all pairs $(\mathcal{O},\mathscr{E})$ where $\mathcal{O} \in \Clu(\bG)$ is a unipotent class and $\mathscr{E}$ is a $\Ql$-constructible local system on $\mathcal{O}$. We will need the notion of a \emph{good pair} as introduced by Geck in \cite[4.4]{geck:1999:character-sheaves-and-GGGRs}. Recall that the Springer correspondence gives us an embedding $\Irr(\bW) \hookrightarrow \mathcal{N}_{\bG}$, (see \cite{lusztig:1984:intersection-cohomology-complexes}), which we denote by $\rho \mapsto (\mathcal{O}_{\rho},\mathscr{E}_{\rho})$. With this we may now define a function $d : \Irr(\bW) \to \mathbb{N}_0$ as follows. Let $\rho \in \Irr(\bW)$ then we then define $d(\rho)$\index{d@$d(\rho)$} to be $\dim\mathfrak{B}_u^{\bG}$ where $\mathfrak{B}_u^{\bG}$\index{BuG@$\mathfrak{B}_u^{\bG}$} is the variety of all Borel subgroups of $\bG$ containing $u \in \mathcal{O}_{\rho}$. We also denote by $b : \Irr(\bW) \to \mathbb{N}_0$\index{b@$b(\rho)$} the function which maps an irreducible character to its $b$-invariant, (see \cite[\S5.2.2]{geck-pfeiffer:2000:characters-of-finite-coxeter-groups}). With this invariant we may define the Lusztig--MacDonald--Spaltenstein induction map $j_{\bW'}^{\bW}$\index{jW@$j_{\bW'}^{\bW}(\rho)$} where $\bW'$ is any subgroup of $\bW$, (see \cite[\S5.2.8]{geck-pfeiffer:2000:characters-of-finite-coxeter-groups}).
\end{pa}

\begin{prop}[{}{Lusztig, \cite[Theorem 1.5]{lusztig:2009:unipotent-classes-and-special-Weyl}}]\label{prop:lusztig-j-ind}
Assume $(\tilde{s},\bW^{\star}(\tilde{\mathcal{F}})) \in \overline{\mathcal{T}}_{\btG}$ is any pair and let $\rho_0 \in \bW^{\star}(\tilde{\mathcal{F}})$ be the unique special character (see \cite[Theorem 6.5.13(b)]{geck-pfeiffer:2000:characters-of-finite-coxeter-groups}) then we have
\begin{equation*}
\Ind_{\bW(\tilde{s})}^{\bW}(\rho_0) = \rho_0' + \text{ a combination of }\tilde{\rho} \in \Irr(\bW) \text{ with }b(\tilde{\rho})> b(\rho_0)
\end{equation*}
where $\rho_0' \in \Irr(\bW)$ satisfies $b(\rho_0') = b(\rho_0)$. Furthermore the pair $(\mathcal{O}_{\rho_0'},\mathscr{E}_{\rho_0'})$ corresponding to $\rho_0'$ under the Springer correspondence satisfies $\mathscr{E}_{\rho_0'} \cong \Ql$.
\end{prop}

\begin{definition}
Recall from \cite[1.1(I)]{spaltenstein:1985:on-the-generalized-springer-correspondence} that we have $b(\rho) \geqslant d(\rho)$ for all $\rho \in \Irr(\bW)$, hence in the notation of \cref{prop:lusztig-j-ind} we have an inequality $b(\tilde{\rho}) \geqslant d(\tilde{\rho}) \geqslant b(\rho_0)$. With this in mind we say a pair $(\tilde{s},\bW^{\star}(\tilde{\mathcal{F}})) \in \overline{\mathcal{T}}_{\btG}$ is \emph{$d$-good} if the following sharper form of \cref{prop:lusztig-j-ind} holds
\begin{equation*}
\Ind_{\bW(\tilde{s})}^{\bW}(\rho_0) = \rho_0' + \text{ a combination of }\tilde{\rho} \in \Irr(\bW) \text{ with }d(\tilde{\rho}) > b(\rho_0).
\end{equation*}
\end{definition}

\begin{pa}
Recall that we are interested in the characters of $G$. These are classified in \cite{lusztig:1988:reductive-groups-with-a-disconnected-centre} by understanding the restriction of characters from $\tilde{G}$ to $G$, in particular we have the following. Assume $\tilde{\psi} \in \mathcal{E}(\tilde{G},\tilde{s})$ is an irreducible character of $\tilde{G}$ and $\psi := \Psi_{\tilde{s}}(\tilde{\psi})$ is the corresponding unipotent character. The quotient group $A_{G^{\star}}(s)$ acts on $\psi$ by conjugation and we denote the stabiliser of $\psi$ under this action by $\Stab_{A_{G^{\star}}(s)}(\psi)$\index{Stab@$\Stab_{A_{G^{\star}}(s)}(\psi)$}, (note here $A_{G^{\star}}(s)$ is defined with respect to $F_s^{\star}$). The main result of \cite{lusztig:1988:reductive-groups-with-a-disconnected-centre} states that the restriction $\Res_G^{\tilde{G}}(\tilde{\psi})$ is multiplicity free and contains $|\Stab_{A_{G^{\star}}(s)}(\psi)|$ number of irreducible constituents. The main result below will be concerned with finding characters $\psi$ such that $|\Stab_{A_{G^{\star}}(s)}(\psi)|$ is maximal.
\end{pa}

\begin{pa}
To understand the meaning of the word maximal in the previous sentence we must recall some results from \cite{taylor:2011:on-unipotent-supports} concerning unipotent classes. The embedding $\iota$ induces a bijection $\Clu(\bG) \to \Clu(\btG)$ between the sets of unipotent conjugacy classes and we will implicitly identify each $\mathcal{O} \in \Clu(\bG)$ with its image $\iota(\mathcal{O})$, (similarly for the unipotent elements of $\bG$ and $\btG$).
\end{pa}

\begin{definition}
Let $\mathcal{O} \in \Clu(\bG)^F$ then we say a class representative $u \in \mathcal{O}^F$ is \emph{well chosen} if $|A_{\bG}(u)^F| = |Z_{\bG}(u)^F||A_{\btG}(u)|$, where $Z_{\bG}(u)$\index{ZGu@$Z_{\bG}(u)$} is the image of $Z(\bG)$ in $A_{\bG}(u)$.
\end{definition}

By \cite[Proposition 2.4]{taylor:2011:on-unipotent-supports} the previous definition makes sense, in particular every class $\mathcal{O} \in \Clu(\bG)^F$ of a simple algebraic group contains a well-chosen class representative and we will assume all class representatives of $F$-stable unipotent classes are well chosen. With this in hand we may now give the main theorem of this paper.

\begin{thm}\label{prop:A}
Assume $\bG$ is a simple algebraic group, $p$ is a good prime for $\bG$ and $\mathcal{O} \in \Clu(\bG)^F$. There exists a pair $(\tilde{s},\bW^{\star}(\tilde{\mathcal{F}})) \in \mathcal{T}_{\btG}$ admitting a unipotent character $\psi \in \mathcal{F}$ satisfying the following properties:
\begin{properties}
	\item $n_{\psi} = |A_{\btG}(u)|$.\label{P1}
	\item $|\Stab_{A_{G^{\star}}(s)}(\psi)| = |Z_{\bG}(u)^F|$.\label{P2}
	\item $j_{\bW(\tilde{s})}^{\bW}(\rho)$ corresponds to $(\mathcal{O},\Ql)$ under the Springer correspondence where $\rho \in \bW^{\star}(\tilde{\mathcal{F}})$ is the unique special character. In particular this ensures that $\Phi_{\tilde{G}}(\Psi_{\tilde{s}}^{-1}(\psi)) = \mathcal{O}$.\label{P3}
\end{properties}
Furthermore the following conditions hold unless $\bG$ is a spin/half-spin group and $A_{\bG}(u)$ is non-abelian:
\begin{properties}[resume]
	\item the pair $(\tilde{s},\bW^{\star}(\tilde{\mathcal{F}}))$ is $d$-good.\label{P4}
	\item $\mathfrak{X}_{\mathcal{F}} := \{\psi \in \mathcal{F} \mid |\Stab_{A_{G^{\star}}(s)}(\psi)| \neq |Z_{\bG}(u)^F|\} = \emptyset$.\label{P5}
\end{properties}
\end{thm}

\begin{rem}
If $\bG$ is adjoint then \cref{P2,P5} are trivially satisfied, so we will only need to concern ourselves with this when $\bG$ has a disconnected centre. Furthermore, although strange in appearance \cref{P5} is important for applications to generalised Gelfand--Graev representations (as is explained in \cite{taylor:2011:on-unipotent-supports}).
\end{rem}

\section{The Setup}\label{sec:the-setup}
\begin{pa}
We now introduce the appropriate notation and machinery that we will need to check \cref{prop:A} effectively. Firstly we fix an algebraic closure $\Ql$ of the field of $\ell$-adic numbers with $\ell$ a prime distinct from $p$ and we assume that any representation of a finite group is taken over $\Ql$. Let us also note here that $\mathbb{N} = \{0,1,2,3,\dots\}$ will denote the set of all natural numbers including 0. We assume fixed an isomorphism of groups $\imath : (\mathbb{Q}/\mathbb{Z})_{p'} \to \mathbb{K}^{\times}$\index{i@$\imath$} and an injective homomorphism of groups $\jmath : \mathbb{Q}/\mathbb{Z} \to \overline{\mathbb{Q}}_{\ell}^{\times}$\index{j@$\jmath$}. Note that $(\mathbb{Q}/\mathbb{Z})_{p'}$ is the subgroup of all elements whose order is coprime to $p$. The composition $\jmath \circ \imath^{-1}$ gives an injective homomorphism $\kappa : \mathbb{K}^{\times} \to \overline{\mathbb{Q}}_{\ell}^{\times}$\index{k@$\kappa$}. As $\bG$ is simple we can and will express the Frobenius endomorphism $F$ as a composition $F_r\circ\tau$\index{Fr@$F_r$}\index{tau@$\tau$} where $\tau$ is a graph automorphism of $\bG$ and $F_r$ is a field automorphism of $\bG$ for some $r$ a power of $p$, (note that $r = q$ if $\tau$ is trivial).
\end{pa}

\begin{pa}\label{pa:root-data}
Let us denote the root datum of $\bG$, (resp.\ $\bG^{\star}$), relative to $\bT_0$, (resp.\ $\bT_0^{\star}$), by $(X,\Phi,\widecheck{X},\widecheck{\Phi})$\index{X@$X$}\index{Phi@$\Phi$}\index{Xcheck@$\widecheck{X}$}\index{Phicheck@$\widecheck{\Phi}$}, (resp.\ $(X^{\star},\Phi^{\star},\widecheck{X}^{\star},\widecheck{\Phi}^{\star})$)\index{Xstar@$X^{\star}$}\index{Phistar@$\Phi^{\star}$}\index{Xcheckstar@$\widecheck{X}^{\star}$}\index{Phicheckstar@$\widecheck{\Phi}^{\star}$}. Here $X := X(\bT_0) = \Hom(\bT_0,\mathbb{K}^{\times})$ and $\widecheck{X} := \widecheck{X}(\bT_0) = \Hom(\mathbb{K}^{\times},\bT_0)$ are the sets of all algebraic group homomorphisms containing respectively the roots $\Phi$ and coroots $\widecheck{\Phi}$ of $\bG$, (similarly in the dual case). We assume $\Delta \subset \Phi$\index{Delta0@$\Delta$} and $\widecheck{\Delta} \subset \widecheck{\Phi}$\index{Delta0check@$\widecheck{\Delta}$} are the sets of simple roots and simple coroots determined by our choice of Borel subgroup $\bB_0$. For each $\alpha \in \Delta$ we denote by $m_{\alpha} \in \mathbb{N}$\index{malpha@$m_{\alpha}$} the natural numbers such that $\alpha_0 = - \sum_{\alpha \in \Delta} m_{\alpha}\alpha \in \Phi$ is the unique lowest root of $\Phi$, (this exists as $\bG$ is simple). We will denote by $\tilde{\Delta} = \Delta \cup \{\alpha_0\}$\index{Delta0tilde@$\tilde{\Delta}$} the set of extended simple roots and by $\mathbb{S}_0 = \{s_{\alpha} \mid \alpha\in\tilde{\Delta}\}$\index{S0@$\mathbb{S}_0$} the corresponding set of reflections. Assume $\tilde{\Delta}^{\star}$\index{Delta0tildestar@$\tilde{\Delta}^{\star}$} is defined with respect to $\Delta^{\star} \subset \Phi^{\star}$\index{Delta0star@$\Delta^{\star}$} in the same way as $\tilde{\Delta}$ is defined with respect to $\Delta$. Then similarly we denote by $\mathbb{T}_0 = \{t_{\alpha} \mid \alpha \in \widetilde{\Delta}^{\star}\}$\index{T0@$\mathbb{T}_0$} the corresponding set of reflections in $\bW^{\star}$.

Denote by $\mathbb{R}X$\index{RX@$\mathbb{R}X$} the $\mathbb{R}$-vector space $\mathbb{R} \otimes_{\mathbb{Z}} X$ and by $\mathbb{R}\widecheck{X}$\index{RXcheck@$\mathbb{R}\widecheck{X}$} the $\mathbb{R}$-vector space $\mathbb{R} \otimes_{\mathbb{Z}} \widecheck{X}$. These spaces have a canonical perfect pairing which we denote by $\langle-,-\rangle : \mathbb{R}X \times \mathbb{R}\widecheck{X} \to \mathbb{R}$\index{langlerangle@$\langle -,- \rangle$}. We denote by $\Omega = \{\varpi_{\alpha} \mid \alpha \in \Delta\} \subset \mathbb{R}X$\index{Omega@$\Omega$}, (resp.\ $\widecheck{\Omega} = \{\widecheck{\varpi}_{\alpha} \mid \alpha \in \Delta\} \subset \mathbb{R}\widecheck{X}$\index{Omegacheck@$\widecheck{\Omega}$}), the basis dual to $\widecheck{\Delta}$, (resp.\ $\Delta$) and define the fundamental group of the root system to be the quotient group $\Pi = \mathbb{Z}\Omega/\mathbb{Z}\Phi$\index{Pi@$\Pi$}, (similarly $\widecheck{\Pi} = \mathbb{Z}\widecheck{\Omega}/\mathbb{Z}\widecheck{\Phi}$\index{Picheck@$\widecheck{\Pi}$} is the dual fundamental group). We will assume that the extended set of roots $\tilde{\Delta}$ is indexed as $\{\alpha_0,\alpha_1,\dots,\alpha_n\}$\index{alphai@$\alpha_i$}, where $n$ is the semisimple rank of $\bG$, with the explicit labelling taken as in \cite[Plate I - IX]{bourbaki:2002:lie-groups-chap-4-6}. For each $1 \leqslant i \leqslant n$ we denote respectively $s_{\alpha_i}$, $t_{\alpha_i}$ $\varpi_{\alpha_i}$, $\widecheck{\varpi}_{\alpha_i}$ and $m_{\alpha_i}$ simply by $s_i$\index{si@$s_i$}, $t_i$\index{ti@$t_i$}, $\varpi_i$\index{pii@$\varpi_i$}, $\widecheck{\varpi}_i$\index{pichecki@$\widecheck{\varpi}_i$} and $m_i$\index{mi@$m_i$}. Following the conventions of Bonnaf\'{e} we let $\widecheck{\varpi}_{\alpha_0} = 0$ and $m_{\alpha_0} = 1$, (see \cite[\S 3.B]{bonnafe:2005:quasi-isolated}), we also denote $s_{\alpha_0} \in \mathbb{S}_0$, (resp.\ $t_{\alpha_0} \in \mathbb{T}_0$), simply by $s_0$, (resp.\ $t_0$).
\end{pa}

\begin{pa}
To $\bG$ and $\bG^{\star}$ we fix simply connected covers $\delta_{\simc} : \bG_{\simc} \to \bG$\index{Gsc@$\bG_{\simc}$}\index{delta1sc@$\delta_{\simc}$} and $\delta_{\simc}^{\star} : \bG_{\ad}^{\star} \to \bG^{\star}$\index{delta1scstar@$\delta_{\simc}^{\star}$}\index{Gadstar@$\bG_{\ad}^{\star}$} and adjoint quotients $\delta_{\ad} : \bG \to \bG_{\ad}$\index{delta1ad@$\delta_{\ad}$}\index{Gad@$\bG_{\ad}$} and $\delta_{\ad}^{\star} : \bG^{\star} \to \bG_{\simc}^{\star}$\index{delta1adstar@$\delta_{\ad}^{\star}$}\index{Gscstar@$\bG_{\simc}^{\star}$}, (note that $\bG_{\ad}^{\star}$ and $\bG_{\simc}^{\star}$ should be interpreted as $(\bG_{\ad})^{\star}$ and $(\bG_{\simc})^{\star}$ respectively). The kernels of these covers are simply the centres of the appropriate groups. If $\bG$ is not of type $\B$ or $\C$ then we will have $\bG_{\ad}^{\star}$ and $\bG_{\simc}$, (resp.\ $\bG_{\simc}^{\star}$ and $\bG_{\ad}$), are isomorphic as they are simply connected, (resp.\ adjoint), groups of the same type. Therefore we may and will take $\bG_{\ad}^{\star} = \bG_{\simc}$ and $\bG_{\simc}^{\star} = \bG_{\ad}$ in this case. We will assume that the isogenies $\delta_{\simc}$, $\delta_{\simc}^{\star}$, $\delta_{\ad}$ and $\delta_{\ad}^{\star}$ are chosen such that the compositions satisfy $\delta_{\ad} \circ \delta_{\simc}=\delta_{\ad}^{\star} \circ \delta_{\simc}^{\star}$. Note that such isogenies always exist as a consequence of the isogeny theorem for algebraic groups, (see for instance the remarks in \cite[\S1.2]{taylor:2012:thesis}). By \cite[Proposition 22.7]{malle-testerman:2011:linear-algebraic-groups} there exist Frobenius endomorphisms of $\bG_{\simc}$ and $\bG_{\ad}$, which we again denote by $F$, such that $\delta_{\simc}$ and $\delta_{\ad}$ are defined over $\mathbb{F}_q$. Similarly we will denote by $F^{\star}$ a Frobenius endomorphism of both $\bG_{\simc}^{\star}$ and $\bG_{\ad}^{\star}$ such that $\delta_{\simc}^{\star}$ and $\delta_{\ad}^{\star}$ are also defined over $\mathbb{F}_q$.
\end{pa}

\begin{pa}
Let us fix a maximal torus and Borel subgroup $\bT_{\simc} \leqslant \bB_{\simc} \leqslant \bG_{\simc}$\index{Tsc@$\bT_{\simc}$}\index{Bsc@$\bB_{\simc}$} then we may assume that $\bT_0 \leqslant \bB_0$ are the images of $\bT_{\simc}\leqslant\bB_{\simc}$ under the isogeny $\delta_{\simc}$. Furthermore we may define a maximal torus and Borel subgroup $\bT_{\ad} \leqslant \bB_{\ad} \leqslant \bG_{\ad}$\index{Tad@$\bT_{\ad}$}\index{Bad@$\bB_{\ad}$} by letting these be the images of $\bT_{\simc} \leqslant \bB_{\simc}$ under the isogeny $\delta_{\ad} \circ \delta_{\simc}$. Similarly we fix a maximal torus and Borel subgroup $\bT_{\ad}^{\star} \leqslant \bB_{\ad}^{\star} \leqslant \bG_{\ad}^{\star}$\index{Tadstar@$\bT_{\ad}^{\star}$}\index{Badstar@$\bB_{\ad}^{\star}$} and we assume $\bT_0^{\star} \leqslant \bB_0^{\star}$ are the images of $\bT_{\ad}^{\star} \leqslant \bB_{\ad}^{\star}$ under the isogeny $\delta_{\simc}^{\star}$. As before we define a maximal torus and Borel subgroup $\bT_{\simc}^{\star} \leqslant \bB_{\simc}^{\star} \leqslant \bG_{\simc}^{\star}$\index{Tscstar@$\bT_{\simc}^{\star}$}\index{Bscstar@$\bB_{\simc}^{\star}$} to be the images of $\bT_{\ad}^{\star} \leqslant \bB_{\ad}^{\star}$ under $\delta_{\ad}^{\star}\circ\delta_{\simc}^{\star}$, (note once again that $\bT_{\ad}^{\star}$ and $\bT_{\simc}^{\star}$ should be interpreted as $(\bT_{\ad})^{\star}$ and $(\bT_{\simc})^{\star}$ respectively -- similarly for the the Borel subgroups). If $\bG$ is not of type $\B$ or $\C$ then we will assume $\bT_{\simc} = \bT_{\ad}^{\star}$ and $\bB_{\simc} = \bB_{\ad}^{\star}$. We now assume that the duality isomorphisms $\varphi_{\ad}$, $\varphi$ and $\varphi_{\simc}$ are chosen such that the following diagram is commutative.

\begin{center}
\begin{tikzcd}
X(\bT_{\ad}) \arrow[hook]{r}\arrow{d}{\varphi_{\ad}}
&X(\bT_0) \arrow[hook]{r}\arrow{d}{\varphi}
&X(\bT_{\simc}) \arrow{d}{\varphi_{\simc}}\\
\widecheck{X}(\bT_{\ad}^{\star}) \arrow[hook]{r}
&\widecheck{X}(\bT_0^{\star}) \arrow[hook]{r}
&\widecheck{X}(\bT_{\simc}^{\star})
\end{tikzcd}
\end{center}

\noindent Here the horizontal maps are those induced by our chosen isogenies.

The above maps give us bijections between the sets of roots and coroots so we will denote by $\Phi$, $\widecheck{\Phi}$ the common set of roots and coroots of $\bG_{\ad}$, $\bG$ and $\bG_{\simc}$; also we write $\Phi^{\star}$, $\widecheck{\Phi}^{\star} $ for the common set of roots and coroots of $\bG_{\simc}^{\star}$, $\bG^{\star}$ and $\bG_{\ad}^{\star}$. If $\bG$ is not of type $\B$ or $\C$ then we will have $X(\bT_{\ad}) = X(\bT_{\simc}^{\star})$ and $\widecheck{X}(\bT_{\ad}) = \widecheck{X}(\bT_{\simc}^{\star})$, which means we also have $\Phi = \Phi^{\star}$ and $\widecheck{\Phi} = \widecheck{\Phi}^{\star}$. The Coxeter systems of $\bG_{\ad}$ and $\bG_{\simc}$ will be identified with $(\bW,\mathbb{S})$ through the above isogenies. Similarly we will identify the Coxeter systems of $\bG_{\simc}^{\star}$ and $\bG_{\ad}^{\star}$ with $(\bW^{\star},\mathbb{T})$ through the above isogenies.
\end{pa}

\section{Classes and Characters}\label{sec:classes-and-chars}
\begin{pa}\label{pa:classes-and-chars}
We set out here the labelling conventions that will be maintained throughout. Assume $\bH$ is a connected reductive algebraic group whose derived subgroup $\bH'$ is simple. Assume that $\bH'$ is of classical type then the elements of $\Clu(\bH)$ will be described in terms of partitions as in \cite[\S13.1]{carter:1993:finite-groups-of-lie-type}. Specifically the partition of $\mathcal{O} \in \Clu(\bH)$ is given by the elementary divisors of $u \in \mathcal{O}$ in a natural matrix representation of $\bH$. If $\bH'$ is of exceptional type then the elements of $\Clu(\bH)$ are described using the Bala--Carter labelling, which is also described in \cite[\S13.1]{carter:1993:finite-groups-of-lie-type}. For the parameterisation of the characters of Weyl groups, (except for the case of $\G_2$), and unipotent characters we will follow the parameterisation defined in \cite[Chapter 4]{lusztig:1984:characters-of-reductive-groups}. In particular if $\bW$ is a Weyl group of type $\B_n$, (resp.\ $\D_n$), then $\Irr(\bW)$ will be parameterised in terms of symbols of rank $n$ and defect 1, (resp.\ defect 0). However, for notational convenience, we will denote the two-row symbol $\bigl[\begin{smallmatrix} A \\ B \end{smallmatrix}\bigr]$ by $[A;B]$\index{AB@$[A;B]$}. In the case of Weyl groups of type $\G_2$ we will follow the labelling given in \cite[\S13.2]{carter:1993:finite-groups-of-lie-type}.

It is clear from \cref{prop:A} that we will also need to know part of the Springer correspondence. The image of the springer correspondence contains the subset $\{(\mathcal{O},\Ql) \mid \mathcal{O} \in \Clu(\bH)\}$ of $\mathcal{N}_{\bH}$ and it is this part of the map that will interest us. If $\bH'$ is of classical type then this part of the Springer correspondence is described combinatorially in \cite[\S2]{geck-malle:2000:existence-of-a-unipotent-support}. If $\bH'$ is of exceptional type then the Springer correspondence is given by the tables in \cite[\S13.3]{carter:1993:finite-groups-of-lie-type}. We will denote by $\rho(\mathcal{O}) \in \Irr(\bW)$\index{rhoO@$\rho(\mathcal{O})$} the character corresponding to the pair $(\mathcal{O},\Ql)$ under the Springer correspondence.

In type $\D_n$ all of the labelling sets considered above have some ambiguity with regard to degenerate elements. Here we say a unipotent class $\mathcal{O} \in \Clu(\bG)$\index{O@$\mathcal{O}_{\pm}$} is degenerate if it is parameterised by a partition $\lambda \vdash 2n$ all of whose entries are even. We say a character of $\bW$ is degenerate if its corresponding symbol is degenerate, (in the sense of \cite[\S4.6]{lusztig:1984:characters-of-reductive-groups}). We will use a $\pm$\index{Lambda@$[\Lambda]_{\pm}$} notation to distinguish between all degenerate elements, note that in \cite[Chapter 4]{lusztig:1984:characters-of-reductive-groups} Lusztig uses the notation $s'$/$s''$. We will make this concrete below by describing explicitly the Springer correspondence in this case.
\end{pa}

\begin{pa}
To prove \cref{prop:A} we will need to be able to discern the action of automorphisms on unipotent characters. In this direction we have the following result which was already known to Lusztig in \cite{lusztig:1988:reductive-groups-with-a-disconnected-centre} but was formalised by Digne--Michel in \cite[Proposition 6.6]{digne-michel:1990:lusztigs-parametrization} and Malle in \cite[\S1]{malle:1991:darstellungstheorie-galoisgruppen}, (see also \cite[Proposition 3.7]{malle:2007:height-0-characters}).
\end{pa}

\begin{lem}\label{lem:invariance-unipotent}
Assume $\bH$ is simple, $F : \bH \to \bH$ is a Frobenius endomorphism and $\gamma$ is an automorphism of $\bH$ commuting with $F$. If $\gamma$ induces the identity on $\bW$ then every unipotent character of $H$ is fixed under composition with $\gamma$. Assume $\bH$ is of type $\A_n$, $\D_n$ or $\E_6$ and $\gamma$ acts on the corresponding Coxeter system $(\bW',\mathbb{S}')$ as a non-trivial graph automorphism. Then every unipotent character is fixed under composition with $\gamma$ except in the following cases:
\begin{itemize}
	\item $\bH$ is of type $\D_{2n}$, $\gamma$ has order 2 and the character is parameterised by a degenerate symbol.
	\item $\bH$ is of type $\D_4$, $\gamma$ has order 3 and the character is parameterised by one of the symbols
\end{itemize}
\begin{equation}\label{eq:stable-chars-D4}
[2;2]_{\pm}
\quad
[12;12]_{\pm}
\quad
[01;14]
\quad
[012;124].
\end{equation}
\end{lem}

\begin{rem}
Recall that we will be interested in the action of $A_{G^{\star}}(s)$ by conjugation on the unipotent characters of $C_{G^{\star}}(s)^{\circ}$ for some semisimple element $s \in \bT_0^{\star}$. Note that each such automorphism is of the form stated in \cref{lem:invariance-unipotent}.
\end{rem}

\section{Explicit Descriptions for Half-Spin Groups}\label{sec:clarification-typeD}
When dealing with the simple algebraic groups of type $\D_n$, with $n \geqslant 4$, we must be quite careful. In this section we assume $\bG$ is such a group. Note that the notational conventions we develop in this section for such groups shall be maintained throughout.

\subsection{Describing Half-Spin Groups}
\begin{pa}
We start by considering precisely the structure of the fundamental group $\Pi = \mathbb{Z}\Omega/\mathbb{Z}\Phi$. From \cite[Plate IV(VIII)]{bourbaki:2002:lie-groups-chap-4-6} we have the fundamental group is given by $\Pi = \{\mathbb{Z}\Phi,\varpi_1 + \mathbb{Z}\Phi, \varpi_{n-1} + \mathbb{Z}\Phi, \varpi_n + \mathbb{Z}\Phi\}$, which is isomorphic to $C_2 \times C_2$ if $n \equiv 0 \pmod{2}$ and $C_4$ if $n \equiv 1 \pmod{2}$, (here $C_m$ is a cyclic group of order $m$). Recall that as $\bG$ is semisimple we have the isomorphism type of $\bG$ is determined by the image of $X$ in the fundamental group. If the image of $X$ is the subgroup generated by $\varpi_1 + \mathbb{Z}\Phi$ then $\bG$ is a special orthogonal group $\SO_{2n}(\mathbb{K})$. If $n \equiv 0 \pmod{2}$ then there are two other non-trivial cases, namely if the image of $X$ is the subgroup generated by $\varpi_{n-1} + \mathbb{Z}\Phi$ or $\varpi_n + \mathbb{Z}\Phi$ then $\bG$ is a half-spin group $\HSpin_{2n}(\mathbb{K})$, (see for instance \cite[\S 7]{carter:1981:centralisers-s/s-classical-groups}). The problem arises here in the choice over the root datum of a half-spin group. Note these groups are isomorphic because there exists an isomorphism of their root data which exchanges the weights $\varpi_{n-1}$ and $\varpi_n$. We will now fix a choice of half-spin group but it will be clear, because of this isomorphism, that the results we prove do not depend upon this choice.
\end{pa}

\begin{assumption}
If $\bG$ is a half-spin group then we assume the image of $X$ in the fundamental group $\Pi$ is $\langle \varpi_n + \mathbb{Z}\Phi \rangle$.
\end{assumption}

\begin{pa}
Let us assume now that $\bG$ is a half-spin group and fix a basis $\{\chi_1,\dots,\chi_n\}$ of $\mathbb{R}X$ such that $\chi_1 = \varpi_n$ and $\chi_i = \alpha_i$ for $2 \leqslant i \leqslant n$. We write $A$ for the change of basis matrix of $\mathbb{R}X$ sending the simple roots $\Delta$ to $\{\chi_1,\dots,\chi_n\}$. This matrix has the form
\begin{equation*}
A = \left[\begin{array}{c|ccc}
a_1 & a_2 & \cdots & a_n\\\hline
0 &  &  & \\
\vdots &  & I_{n-1} & \\
0 &  &  & 
\end{array}\right],
\end{equation*}
where $(a_1,\dots,a_n) = (\frac{1}{2},1,\frac{3}{2},2,\dots,\frac{(n-2)}{2},\frac{(n-2)}{4},\frac{n}{4})$ and $I_{n-1}$ is the $(n-1) \times (n-1)$ identity matrix. To our chosen basis $\{\chi_1,\dots,\chi_n\}$ of $\mathbb{R}X$ we have a dual basis $\{\gamma_1,\dots,\gamma_n\}$ of $\mathbb{R}\widecheck{X}$. Let $B$ be the change of basis matrix of $\mathbb{R}\widecheck{X}$ sending the simple coroots $\widecheck{\Delta}$ to $\{\gamma_1,\dots,\gamma_n\}$. The matrices $A$ and $B$ satisfy the condition $ACB^T = I_n$ where $C = (\langle \alpha_i, \widecheck{\alpha}_j \rangle)_{1 \leqslant i,j \leqslant n}$ is the \emph{Cartan matrix} and $I_n$ is the $n \times n$ identity matrix.

Let us now consider the root datum of the associated dual group $\bG^{\star}$. If $X$ has image $\langle \varpi_n + \mathbb{Z}\Phi \rangle$ in $\Pi$ then for $X$ to be isomorphic to $\widecheck{X}^{\star}$, (and for such an isomorphism to preserve the pairing $\langle - , - \rangle$), we must have $\widecheck{X}^{\star}$ has image $\langle \widecheck{\varpi}_n + \mathbb{Z}\widecheck{\Phi} \rangle$ in $\widecheck{\Pi}$. From this we can easily calculate the image of $\widecheck{X}$ in $\widecheck{\Pi}$ as follows. Recall that the matrix $B$ expresses the decomposition of the basis of $\mathbb{R}\widecheck{X}$ in terms of the simple coroots. The Cartan matrix expresses the decomposition of the simple coroots in terms of the fundamental dominant coweights hence $BC = A^{-T}$ gives the decomposition of $\{\gamma_1,\dots,\gamma_n\}$ in terms of $\{\widecheck{\varpi}_1,\dots,\widecheck{\varpi}_n\}$. We easily determine that this matrix has the form
\begin{equation*}
A^{-T} = \left[\begin{array}{c|ccc}
a_1' & 0 & \cdots & 0\\\hline
a_2' &  &  & \\
\vdots &  & I_{n-1} & \\
a_n' &  &  & 
\end{array}\right],
\end{equation*}
where $(a_1',\dots,a_n') = (2,-2,-3,\dots,-(n-2),-\frac{(n-2)}{2},-\frac{n}{2})$.

From this we see that the image of $\widecheck{X}$ in $\widecheck{\Pi}$ is determined by the image of $\gamma_n$ in $\widecheck{\Pi}$, i.e.\ the element $\frac{n}{2}\widecheck{\varpi}_1 + \widecheck{\varpi}_n + \mathbb{Z}\widecheck{\Phi}$. Therefore we have the image of $\widecheck{X}$ in $\widecheck{\Pi}$ is $\langle \widecheck{\varpi}_n + \mathbb{Z}\widecheck{\Phi} \rangle$ if $n \equiv 0 \pmod{4}$ and $\langle \widecheck{\varpi}_{n-1} + \mathbb{Z}\widecheck{\Phi} \rangle$ if $n \equiv 2 \pmod{4}$. In particular, using the dual isomorphism $\widecheck{X} \to X^{\star}$, we must have the image of $X^{\star}$ in $\Pi$ is $\langle \varpi_n + \mathbb{Z}\Phi \rangle$ if $n \equiv 0 \pmod{4}$ and $\langle \varpi_{n-1} + \mathbb{Z}\Phi \rangle$ if $n \equiv 2 \pmod{4}$. As a consequence we see that the dual group of a half-spin group is again isomorphic to a half-spin group but its root datum depends upon $n$. It will be useful for us to also describe $\Ker(\delta_{\simc}^{\star})$ but to do this we need the following lemma.
\end{pa}

\begin{lem}[{}{see \cite[Proposition 4.1]{bonnafe:2006:sln}}]\label{lem:centre-fund-grp}
Let $\bG$ be a connected semisimple algebraic group. There exists a canonical surjective homomorphism $\mathbb{Q} \otimes_{\mathbb{Z}} \widecheck{X}(\bT_0) \to \bT_0$ which induces an isomorphism $(\mathbb{Z}\widecheck{\Omega}/\widecheck{X})_{p'} \to Z(\bG)$.
\end{lem}

Note that the above isomorphism depends upon the choice of $\imath$. Taking the above lemma in the case where $\bG$ is simply connected of type $\D_n$ this says we have a natural isomorphism $\widecheck{\Pi} \overset{\sim}{\longrightarrow} Z(\bG)$, (recall we assume $p \neq 2$). We now make the following convention regardless of the congruence of $n \pmod{2}$.

\begin{assumption}
Assume $\bG$ is a simply connected group of type $\D_n$ then we denote the centre of $\bG$ by $Z(\bG) = \{1,\hat{z}_1,\hat{z}_{n-1},\hat{z}_n\}$\index{zihat@$\hat{z}_i$}. We fix the notation such that $\widecheck{\varpi}_{n-1} + \mathbb{Z}\widecheck{\Phi} \mapsto \hat{z}_{n-1}$ and $\widecheck{\varpi}_n + \mathbb{Z}\widecheck{\Phi} \mapsto \hat{z}_n$ under the isomorphism specified by \cref{lem:centre-fund-grp}.
\end{assumption}

\noindent Under this convention whenever $\bG$ is a half-spin group we will have $\Ker(\delta_{\simc}^{\star}) = \langle \hat{z}_n \rangle$, however $\Ker(\delta_{\simc})$ will be $\langle \hat{z}_n \rangle$ if $n \equiv 0 \pmod{4}$ and $\langle \hat{z}_{n-1} \rangle$ if $n \equiv 2 \pmod{4}$. Furthermore whenever $\bG$ is a special orthogonal group we will have $\Ker(\delta_{\simc}^{\star}) = \Ker(\delta_{\simc}) = \langle \hat{z}_1 \rangle$.

\subsection{The Springer Correspondence}
\begin{pa}\label{pa:degen-springer-corr}
In this section we wish to remove the ambiguity over the labelling of elements in $\Clu(\bG)$ and $\Irr(\bW)$ and in particular express concretely the Springer correspondence for degenerate elements. Recall that for this situation to occur we must necessarily have $n \equiv 0 \pmod{2}$. To clarify the Springer correspondence we will use the argument given in \cite[\S13.3]{carter:1993:finite-groups-of-lie-type}.

Let $\lambda \vdash 2n$ be a degenerate partition then we can express $\lambda$ as $(2\eta_1,2\eta_1,\dots,2\eta_s,2\eta_s)$, where $s$, $\eta_i \in \mathbb{N}$. We denote by $\eta$ the sequence $(\eta_1,\dots,\eta_s)$ then $\eta$ is a partition of $n/2$. If $\xi = (\xi_1,\dots,\xi_r) \vdash n/2$ is a partition of $n/2$ then there are two $\bG$-conjugacy classes of Levi subgroups with semisimple type $\A_{2\xi_1-1}\cdots\A_{2\xi_r-1}$. Two class representatives can be given by standard Levi subgroups and the two possibilities depend on whether the root $\alpha_n$ or $\alpha_{n-1}$ is contained in the root system of the standard Levi. We will denote by $\bL_{\xi}^+$\index{Lxi@$\bL_{\xi}^{\pm}$} the Levi subgroup whose root system contains $\alpha_{n-1}$ and $\bL_{\xi}^-$ the Levi subgroup whose root system contains $\alpha_n$. By the Bala--Carter theorem, (see \cite[Theorem 5.9.5]{carter:1993:finite-groups-of-lie-type}), if $\mathcal{O} \in \Clu(\bG)$ is a unipotent class parameterised by the degenerate partition $\lambda$ then either $\mathcal{O} \cap \bL_{\eta}^+$ contains the regular unipotent class of $\bL_{\eta}^+$ or $\mathcal{O} \cap \bL_{\eta}^-$ contains the regular unipotent class of $\bL_{\eta}^-$.

We now turn to the irreducible characters of $\bW$. Assume $[\Lambda]_{\pm} \in \Irr(\bW)$ are the two irreducible characters parameterised by the degenerate symbol $[\Lambda]$. Adopting the notation above we denote the Weyl groups of the standard Levi subgroups $\bL_{\xi}^{\pm}$ by $\bW(\A_{\xi}^{\pm})$. By \cite[Theorem 5.4.5]{geck-pfeiffer:2000:characters-of-finite-coxeter-groups} and \cite[Proposition 5.6.3]{geck-pfeiffer:2000:characters-of-finite-coxeter-groups} there is a unique partition $\xi \vdash n/2$ such that $\{[\Lambda]_+,[\Lambda]_-\} = \{j_{\bW(\A_{\xi^*}^+)}^{\bW}(\sgn),j_{\bW(\A_{\xi^*}^-)}^{\bW}(\sgn)\}$ where $\sgn \in \Irr(\bW(\A_{\xi^*}^{\pm}))$ denotes the sign character and $\xi^*$ denotes the dual partition. With this we now distinguish the degenerate objects in the following way.
\end{pa}

\begin{assumption}
We assume the $\pm$ convention to be chosen such that $[\Lambda]_{\pm} = j_{\bW(\A_{\xi^*}^{\pm})}^{\bW}(\sgn)$\index{Lambda@$[\Lambda]_{\pm}$} for some (unique) partition $\xi \vdash n/2$. Furthermore $\mathcal{O}_{\lambda}^{\pm} \in \Clu(\bG)$\index{O@$\mathcal{O}_{\pm}$} is the (unique) unipotent class such that $\mathcal{O}_{\lambda}^{\pm} \cap \bL_{\eta}^{\pm}$ contains the regular unipotent class of $\bL_{\eta}^{\pm}$, with $\eta$ as above.
\end{assumption}

\begin{pa}
We now come to an interesting dichotomy, (which is the duality discussed by Spaltenstein in \cite[Chapitre III]{spaltenstein:1982:classes-unipotentes}). The way we have identified the two degenerate unipotent classes will allow us to compute the order of the component groups of their centraliser, however it will not allow us to compute the Springer correspondence. For this we must identify these classes as Richardson classes associated to their canonical parabolic subgroup. Let us denote by $\eta^*$ the partition of $n/2$ dual to $\eta$ then we have the following result.
\end{pa}

\begin{prop}[{}{see \cite[Proposition II.7.6]{spaltenstein:1982:classes-unipotentes}}]\label{cor:richardson-degen-classes}
The unipotent classes $\mathcal{O}_{\lambda}^{\pm}$ are Richardson classes for parabolic subgroups with Levi complement $\bL_{\eta^*}^{\pm}$ if $n \equiv 0 \pmod{4}$ and $\bL_{\eta^*}^{\mp}$ if $n \equiv 2\pmod{4}$. Furthermore, let us denote by $(n_{\alpha_{n-1}}^{\pm},n_{\alpha_n}^{\pm})$ the weights of the weighted Dynkin diagram of $\mathcal{O}_{\lambda}^{\pm}$ associated to the nodes $\alpha_{n-1}$ and $\alpha_n$. Then we have $(n_{\alpha_{n-1}}^+,n_{\alpha_n}^+) = (a,b)$ and $(n_{\alpha_{n-1}}^-,n_{\alpha_n}^-) = (b,a)$ where $b = 2 - a$ and
\begin{equation*}
a = \begin{cases}
0 &\text{if }n\equiv 0 \pmod{4},\\
2 &\text{if }n\equiv 2 \pmod{4}.
\end{cases}
\end{equation*}
\end{prop}

\begin{rem}
The statement concerning the weighted Dynkin diagram follows from \cite[Proposition II.7.6]{spaltenstein:1982:classes-unipotentes} because the classes are even so they are Richardson classes for their canonical parabolic subgroups, (see \cite[\S7.9 - Proposition]{humphreys:1995:conjugacy-classes}).
\end{rem}

\begin{pa}
Assume $\bL$ is a Levi subgroup of $\bG$ contained in a parabolic subgroup $\bP$ with unipotent radical $\bU_{\bP}$. In \cite{lusztig-spaltenstein:1979:induced-unipotent-classes} Lusztig and Spaltenstein have defined an induction map $\Ind_{\bL}^{\bG}$ taking a unipotent conjugacy class of $\bL$ to a unipotent conjugacy class of $\bG$, which is defined in the following way. If $\mathcal{O}$ is a unipotent class of $\bL$ then $\Ind_{\bL}^{\bG}(\mathcal{O})$ is the unique unipotent conjugacy class of $\bG$ such that $\Ind_{\bL}^{\bG}(\mathcal{O}) \cap \mathcal{O}\bU_{\bP}$ is dense in $\mathcal{O}\bU_{\bP}$. They show that this does not depend on the choice of $\bP$ and depends only on the pair $(\bL,\mathcal{O})$ up to $\bG$ conjugacy. Hence we may assume that $\bL$ is a standard Levi subgroup of $\bG$. Note that the statements in \cite{lusztig-spaltenstein:1979:induced-unipotent-classes} have some restrictions but these were removed in \cite{lusztig:1984:intersection-cohomology-complexes}.
\end{pa}

\begin{prop}[Lusztig and Spaltenstein, {\cite[Theorem 3.5]{lusztig-spaltenstein:1979:induced-unipotent-classes}}]\label{prop:lusztig-spaltenstein}
Assume $\mathcal{O}$ is a unipotent conjugacy class of a Levi subgroup $\bL$ of $\bG$ containing $\bT_0$. Write $\rho(\mathcal{O}) \in \Irr(\bW(\bL))$ for the character corresponding to $(\mathcal{O},\Ql)$ under the Springer correspondence of $\bL$. Let $\tilde{\mathcal{O}} = \Ind_{\bL}^{\bG}(\mathcal{O})$ be the induced class and write $\rho(\tilde{\mathcal{O}}) \in \Irr(\bW)$ for the character corresponding to $(\tilde{\mathcal{O}},\Ql)$ under the Springer correspondence of $\bG$ then $\rho(\tilde{\mathcal{O}}) = j_{\bW(\bL)}^{\bW}(\rho(\mathcal{O}))$.
\end{prop}

\begin{pa}
Using \cref{cor:richardson-degen-classes} and \cite[Proposition 1.9(b)]{lusztig-spaltenstein:1979:induced-unipotent-classes} we have $\mathcal{O}_{\lambda}^{\pm}$ is $\Ind_{\bL_{\eta^*}^{\pm}}^{\bG}(\mathcal{O}_{0})$ if $n \equiv 0 \pmod{4}$ and $\Ind_{\bL_{\eta^*}^{\mp}}^{\bG}(\mathcal{O}_{0})$ if $n \equiv 2 \pmod{4}$, where $\mathcal{O}_{0}$ denotes the trivial unipotent class. Note that there is a restriction on \cite[Proposition 19.(b)]{lusztig-spaltenstein:1979:induced-unipotent-classes} that $p$ is sufficiently large but this is only to ensure that unipotent classes are parameterised by their weighted Dynkin diagrams, which is known to hold in good characteristic. The Springer character of the trivial class is always the sign character, therefore we have
\begin{equation}\label{eq:degen-spring-correspondence}
\rho(\mathcal{O}_{\lambda}^{\pm}) = \begin{cases}
j_{\bW(\A_{\eta^*}^{\pm})}^{\bW}(\sgn) &\text{if }n\equiv 0\pmod{4},\\
j_{\bW(\A_{\eta^*}^{\mp})}^{\bW}(\sgn) &\text{if }n\equiv 2\pmod{4}.
\end{cases}
\end{equation}
This now concretely specifies the Springer correspondence in the degenerate case.
\end{pa}

\subsection{Component Groups in Half-Spin Groups}
\begin{pa}
We now come to the determination of $|A_{\bG}(u^{\pm})|$ where $u^{\pm} \in \mathcal{O}_{\lambda}^{\pm}$ are class representatives for degenerate unipotent classes. From the description of the component groups given in \cite[\S14.3]{lusztig:1984:intersection-cohomology-complexes}, \cite[\S10.6]{lusztig:1984:intersection-cohomology-complexes} and \cite[\S13.1]{carter:1993:finite-groups-of-lie-type} we have $|A_{\bG}(u^+)| = |A_{\bG}(u^-)|$ except when $\bG$ is a half-spin group, in which case we always have $|A_{\bG}(u^+)| \neq |A_{\bG}(u^-)|$. Assume $\bG$ is a half-spin group then we claim that
\begin{equation*}
\begin{gathered}
|A_{\bG}(u^+)| = \begin{cases}
2 &\text{if }n\equiv0\pmod{4},\\
1 &\text{if }n\equiv2\pmod{4},
\end{cases}
\end{gathered}
\qquad
\begin{gathered}
|A_{\bG}(u^-)| = \begin{cases}
1 &\text{if }n\equiv0\pmod{4},\\
2 &\text{if }n\equiv2\pmod{4},
\end{cases}
\end{gathered}
\end{equation*}
Let us now verify this claim. Firstly let $u_{\ad}^{\pm} = \delta_{\ad}(u^{\pm})$ be corresponding elements in the adjoint group then we have $|A_{\bG_{\ad}}(u_{\ad}^{\pm})| = 1$, (see \cite[\S13.1]{carter:1993:finite-groups-of-lie-type}). In particular we must have $A_{\bG}(u^{\pm}) = Z_{\bG}(u^{\pm})$ hence $|A_{\bG}(u^{\pm})| \in \{1,2\}$ depending upon $Z_{\bG}(u^{\pm})$.

The intersection $\mathcal{O}_{\lambda}^{\pm} \cap \bL_{\eta}^{\pm}$ contains the regular unipotent class of $\bL_{\eta}^{\pm}$ therefore we may take $u^{\pm} \in \mathcal{O}_{\lambda}^{\pm}$ to be such that it is a regular unipotent element of $\bL_{\eta}^{\pm}$. We have a natural embedding $C_{\bL_{\eta}^{\pm}}(u^{\pm}) \to C_{\bG}(u^{\pm})$, which induces an embedding $A_{\bL_{\eta}^{\pm}}(u^{\pm}) \to A_{\bG}(u^{\pm})$. As $u^{\pm}$ is a regular unipotent element we have $A_{\bL_{\eta}^{\pm}}(u^{\pm}) \cong \mathcal{Z}(\bL_{\eta}^{\pm})$ where $\mathcal{Z}(\bL_{\eta}^{\pm})$ denotes the component group $Z(\bL_{\eta}^{\pm})/Z(\bL_{\eta}^{\pm})^{\circ}$, (see for example the proof of \cite[Proposition 14.24]{digne-michel:1991:representations-of-finite-groups-of-lie-type}). Hence to determine whether $|A_{\bG}(u^{\pm})| = 2$ or 1 it is enough to determine when $|\mathcal{Z}(\bL_{\eta}^{\pm})| = 2$ or 1.

To do this calculation we will use a result of Digne--Lehrer--Michel. Recall that we have a natural embedding $Z(\bG) \to Z(\bL_{\eta}^{\pm})$, which induces a surjective map $Z(\bG) \to \mathcal{Z}(\bL_{\eta}^{\pm})$ by \cite[Proposition 4.2]{bonnafe:2006:sln}. The kernel of this surjective map is given to us by the following result.
\end{pa}

\begin{prop}[Digne--Lehrer--Michel, {\cite[Proposition 4.5]{bonnafe:2006:sln}}]\label{prop:dlm-centres}
Let $\mathbb{I} \subset \Delta$ be a set of simple roots and $\bL_{\mathbb{I}}$ the standard Levi subgroup corresponding to $\mathbb{I}$. The kernel of the map $Z(\bG) \to \mathcal{Z}(\bL_{\mathbb{I}})$ is the image of $\langle \widecheck{\varpi}_{\alpha} + \widecheck{X} \mid \alpha \in \Delta\setminus \mathbb{I}\rangle$ under the isomorphism $(\mathbb{Z}\widecheck{\Omega}/\widecheck{X})_{p'} \cong Z(\bG)$ of \cref{lem:centre-fund-grp}.
\end{prop}

\begin{pa}
Recall from \cref{sec:clarification-typeD} that the image of $\widecheck{X}$ in $\widecheck{\Pi}$ depends upon the congruence of $n \pmod{4}$. We treat the two cases separately.
\begin{itemize}
	\item $n \equiv 0 \pmod{4}$ then $\widecheck{X} = \langle \widecheck{\varpi}_n + \mathbb{Z}\widecheck{\Phi} \rangle$. By \cref{prop:dlm-centres} we have the kernel of the map $Z(\bG) \to \mathcal{Z}(\bL_{\eta}^{\pm})$ is non-trivial whenever $\alpha_{n-1}$ is \emph{not} in the root system of the Levi. If the kernel is non-trivial then the order of $\mathcal{Z}(\bL_{\eta}^{\pm})$ is 1. Hence we have $|\mathcal{Z}(\bL_{\eta}^+)| = 2$ and $|\mathcal{Z}(\bL_{\eta}^-)|=1$.
	\item $n \equiv 2 \pmod{4}$ then $\widecheck{X} = \langle \widecheck{\varpi}_{n-1} + \mathbb{Z}\widecheck{\Phi} \rangle$. By \cref{prop:dlm-centres} we have the kernel of the map $Z(\bG) \to \mathcal{Z}(\bL_{\eta}^{\pm})$ is non-trivial whenever $\alpha_n$ is \emph{not} in the root system of the Levi. If the kernel is non-trivial then the order of $\mathcal{Z}(\bL_{\eta}^{\pm})$ is 1. Hence we have $|\mathcal{Z}(\bL_{\eta}^+)| = 1$ and $|\mathcal{Z}(\bL_{\eta}^-)|=2$.
\end{itemize}
This now verifies the statements regarding the component group orders.
\end{pa}

\section{Quasi-Isolated Semisimple Elements}\label{sec:quasi-isolated}
The results we prove in this section will be stated in terms of $\bG$, for notational convenience, but they will be applied to the dual group. The exception to this will be in \Crefrange{sec:comp-grp-s/s-elmt}{subsec:groups-type-Dn} and \cref{sec:fixing-semisimple-class-reps} where we prove results concerning the relationship between $\bG$ and $\bG^{\star}$.

\subsection{\texorpdfstring{Bonnaf\'{e}'s Classification}{Bonnafe's Classification}}
\begin{pa}
We start by describing Bonnaf\'{e}'s classification of quasi-isolated semisimple elements, (see \cite{bonnafe:2005:quasi-isolated}). Assume $\bT \leqslant \bG$ is a maximal torus and recall that we have an isomorphism $\mathbb{K}^{\times} \otimes_{\mathbb{Z}} \widecheck{X}(\bT) \to \bT$ given by $k\otimes \gamma \mapsto \gamma(k)$. Using the isomorphism $\imath : (\mathbb{Q/Z})_{p'} \to \mathbb{K}^{\times}$ we obtain an isomorphism $\imath_{\bT} : (\mathbb{Q/Z})_{p'} \otimes_{\mathbb{Z}} \widecheck{X}(\bT) \to \bT$\index{iT@$\imath_{\bT}$} given by $\imath_{\bT}(r\otimes\gamma) = \gamma(\imath(r))$. We have an action of $F$ on $(\mathbb{Q/Z})_{p'} \otimes_{\mathbb{Z}} \widecheck{X}(\bT)$ given by $F(r\otimes\gamma) = r\otimes F(\gamma)$ which is compatible with the action of $F$ on $\bT$.

As $\bG_{\ad}$ is adjoint the cocharacter group $\widecheck{X}(\bT_{\ad})$ can be identified with the coweight lattice, which means we can naturally consider all fundamental dominant coweights $\widecheck{\varpi}_{\alpha} \in \widecheck{\Omega}$ to be elements of $\widecheck{X}(\bT_{\ad})$. Let $\mathcal{A} := \Aut_{\bW}(\tilde{\Delta}) = \{x \in \bW \mid x(\tilde{\Delta}) = \tilde{\Delta}\} \leqslant \bW$\index{A@$\mathcal{A}$} be the automorphism group of the extended Dynkin diagram in $\bW$ and let $\mathcal{Q}(\bG_{\ad})$\index{QGad@$\mathcal{Q}(\bG_{\ad})$} denote the set of subsets $\Sigma \subset \tilde{\Delta}$ such that the stabiliser of $\Sigma$ in $\mathcal{A}$ acts transitively on $\Sigma$. We then have the following theorem of Bonnaf\'{e}, (recall that we assume here that $p$ is a good prime for $\bG$).
\end{pa}

\begin{thm}[Bonnaf\'{e}, {\cite[Theorem 5.1]{bonnafe:2005:quasi-isolated}}]\label{thm:bonnafe}
Let $\Sigma \in \mathcal{Q}(\bG_{\ad})$ and define an element $t_{\Sigma} \in \bT_{\ad}$ by setting
\begin{equation*}
t_{\Sigma} = \imath_{\bT_{\ad}}\left(\sum_{\alpha \in \Sigma} \frac{1}{m_{\alpha}|\Sigma|}\otimes\widecheck{\varpi}_{\alpha}\right)\index{tsigma@$t_{\Sigma}$},
\end{equation*}
where $\widecheck{\varpi}_{\alpha} \in \widecheck{\Omega}$, (and $m_{\alpha}$ is as in \cref{pa:root-data}). The following then hold:
\begin{itemize}
	\item the map $\Sigma \mapsto t_{\Sigma}$ induces a bijection between the set of orbits of $\mathcal{A}$ acting on $\mathcal{Q}(\bG_{\ad})$ and the set of conjugacy classes of quasi-isolated semisimple elements in $\bG_{\ad}$.
	\item for any $\Sigma \in \mathcal{Q}(\bG_{\ad})$ we have:
	\begin{itemize}
		\item $\bW(t_{\Sigma})^{\circ} = \langle s_{\alpha} \in \mathbb{S}_0 \mid \alpha \in \tilde{\Delta} - \Sigma \rangle$;
		\item $A_{\bG_{\ad}}(t_{\Sigma}) = \{x\bW(t_{\Sigma})^{\circ} \mid x \in \mathcal{A}$ and $x(\Sigma) = \Sigma\}$.
	\end{itemize}
\end{itemize}
\end{thm}

\begin{rem}
In the statement of the above theorem we have identified $A_{\bG}(s)$ and $\bW(s)/\bW(s)^{\circ}$ under the usual natural isomorphism between these two groups, (see for instance \cite[Proposition 1.3(d)]{bonnafe:2005:quasi-isolated}). We will maintain this identification throughout.
\end{rem}

\begin{pa}
An important aspect of Bonnaf\'{e}'s theorem is that he determines the structure of $A_{\bG}(t_{\Sigma})$, which is important to us in verifying the validity of \cref{P2}. In \cref{tab:quasi-isolated-classical,tab:quasi-isolated-exceptional} we reproduce Bonnaf\'{e}'s classification of quasi-isolated semisimple elements in classical and exceptional adjoint algebraic groups, as found in \cite[Tables 2 and 3]{bonnafe:2005:quasi-isolated}. In the case of $\G_2$, $\F_4$ and $\E_8$ the notion of isolated and quasi-isolated semisimple elements coincide as the adjoint and simply connected groups coincide. Note that in the original table of Bonnaf\'{e} the class representative for the class corresponding to $\{\alpha_{n/2}\}$ in $\D_n$ is denoted as having $m_{\alpha}|\Sigma| = 4$. However it is clear that this element has $m_{\alpha}|\Sigma| = 2$ as it is isolated, (see \cite[Proposition 5.5]{bonnafe:2005:quasi-isolated}).
\end{pa}

\begin{table}[t]
\centering
\begin{tabular}{>{$}c<{$}>{$}l<{$}>{$}c<{$}>{$}c<{$}>{$}c<{$}cc}
\toprule
\bG_{\ad} & \multicolumn{1}{c}{$\Sigma$} & m_{\alpha}|\Sigma| & C_{\bG_{\ad}}(t_{\Sigma})^{\circ} & |A_{\bG_{\ad}}(t_{\Sigma})| & Isolated? \tabularnewline
\midrule
\A_n
&
\underset{d \mid n+1\text{ and }p \nmid d}{\{\alpha_{j(n+1)/d} \mid 0 \leqslant j \leqslant d-1\}}
&
d
&
(\A_{(n+1-d)/d})^d
&
d
&
$d=1$
\\
\midrule
\multirow{3}{*}{$\B_n$}
&
\{\alpha_0\}
&
1
&
\B_n
&
1
&
yes
\\
&
\{\alpha_0,\alpha_1\}
&
2
&
\B_{n-1}
&
2
&
no
\\
&
\{\alpha_d\}, d \in [2,n]
&
2
&
\D_d\B_{n-d}
&
2
&
yes
\\
\midrule
\multirow{5}{*}{$\C_n$}
&
\{\alpha_0\}
&
1
&
\C_n
&
1
&
yes
\\
&
\{\alpha_d\},\, d \in [1,n-1]\setminus\{n/2\}
&
2
&
\C_d\C_{n-d}
&
1
&
yes
\\
&
\{\alpha_{n/2}\},\text{ (only if }2\mid n\text{)}
&
2
&
\C_{n/2}\C_{n/2}
&
2
&
yes
\\
&
\{\alpha_0,\alpha_n\}
&
2
&
\A_{n-1}
&
2
&
no
\\
&
\{\alpha_d,\alpha_{n-d}\},\, 1 \leqslant d < n/2
&
4
&
\C_d\A_{n-2d-1}\C_d
&
2
&
no
\\
\midrule
\multirow{8}{*}{$\D_n$}
&
\{\alpha_0\}
&
1
&
\D_n
&
1
&
yes
\\
&
\{\alpha_d\},\, d \in [2,n-2]\setminus\{n/2\}
&
2
&
\D_d\D_{n-d}
&
2
&
yes
\\
&
\{\alpha_{n/2}\}\text{ (only if }2\mid n\text{)}
&
2
&
\D_{n/2}\D_{n/2}
&
4
&
yes
\\
&
\{\alpha_d,\alpha_{n-d}\},\, d \in [2,n-2]\setminus\{n/2\}
&
4
&
\D_d\A_{n-2d-1}\D_d
&
4
&
no
\\
&
\{\alpha_0,\alpha_1,\alpha_{n-1},\alpha_n\}
&
4
&
\A_{n-3}
&
4
&
no
\\
&
\{\alpha_0,\alpha_1\}
&
2
&
\D_{n-1}
&
2
&
no
\\
&
\{\alpha_0,\alpha_{n-1}\}\text{, (only if }2\mid n\text{)}
&
2
&
\A_{n-1}
&
2
&
no
\\
&
\{\alpha_0,\alpha_n\}\text{, (only if }2\mid n\text{)}
&
2
&
\A_{n-1}
&
2
&
no
\\
\bottomrule
\end{tabular}
\caption{Classes of Quasi-Isolated Semisimple Elements in Classical Groups}
\label{tab:quasi-isolated-classical}
\end{table}
\begin{table}[t]
\centering
\begin{tabular}{>{$}c<{$}>{$}l<{$}>{$}c<{$}>{$}c<{$}c}
\toprule
\bG_{\ad} & \multicolumn{1}{c}{$\Sigma$} & C_{\bG_{\ad}}(t_{\Sigma})^{\circ} & |A_{\bG}(t_{\Sigma})| & Isolated? \tabularnewline
\midrule
\multirow{3}{*}{$\G_2$}
&
\{\alpha_0\}
&
\G_2
&
1
&
yes
\\
&
\{\alpha_1\}
&
\A_1\A_1
&
1
&
yes
\\
&
\{\alpha_2\}
&
\A_2
&
1
&
yes
\\
\midrule
\multirow{5}{*}{$\F_4$}
&
\{\alpha_0\}
&
\F_4
&
1
&
yes
\\
&
\{\alpha_1\}
&
\A_1\C_3
&
1
&
yes
\\
&
\{\alpha_2\}
&
\A_2\A_2
&
1
&
yes
\\
&
\{\alpha_3\}
&
\A_3\A_1
&
1
&
yes
\\
&
\{\alpha_4\}
&
\B_4
&
1
&
yes
\\
\midrule
\multirow{5}{*}{$\E_6$}
&
\{\alpha_0\}
&
\E_6
&
1
&
yes
\\
&
\{\alpha_2\}
&
\A_5\A_1
&
1
&
yes
\\
&
\{\alpha_4\}
&
\A_2\A_2\A_2
&
3
&
yes
\\
&
\{\alpha_0,\alpha_1,\alpha_6\}
&
\D_4
&
3
&
no
\\
&
\{\alpha_2,\alpha_3,\alpha_5\}
&
\A_1\A_1\A_1\A_1
&
3
&
no
\\
\midrule
\multirow{8}{*}{$\E_7$}
&
\{\alpha_0\}
&
\E_7
&
1
&
yes
\\
&
\{\alpha_1\}
&
\A_1\D_6
&
1
&
yes
\\
&
\{\alpha_2\}
&
\A_7
&
2
&
yes
\\
&
\{\alpha_3\}
&
\A_2\A_5
&
1
&
yes
\\
&
\{\alpha_4\}
&
\A_3\A_3\A_1
&
2
&
yes
\\
&
\{\alpha_0,\alpha_7\}
&
\E_6
&
2
&
no
\\
&
\{\alpha_1,\alpha_6\}
&
\D_4\A_1\A_1
&
2
&
no
\\
&
\{\alpha_3,\alpha_5\}
&
\A_2\A_2\A_2
&
2
&
no
\\
\midrule
\multirow{9}{*}{$\E_8$}
&
\{\alpha_0\}
&
\E_8
&
1
&
yes
\\
&
\{\alpha_1\}
&
\D_8
&
1
&
yes
\\
&
\{\alpha_2\}
&
\A_8
&
1
&
yes
\\
&
\{\alpha_3\}
&
\A_1\A_7
&
1
&
yes
\\
&
\{\alpha_4\}
&
\A_2\A_1\A_5
&
1
&
yes
\\
&
\{\alpha_5\}
&
\A_4\A_4
&
1
&
yes
\\
&
\{\alpha_6\}
&
\D_5\A_3
&
1
&
yes
\\
&
\{\alpha_7\}
&
\E_6\A_2
&
1
&
yes
\\
&
\{\alpha_8\}
&
\E_7\A_1
&
1
&
yes
\\
\bottomrule
\end{tabular}
\caption{Classes of Quasi-Isolated Semisimple Elements in Exceptional Groups}
\label{tab:quasi-isolated-exceptional}
\end{table}

\subsection{\texorpdfstring{The Group $\mathcal{A}$}{The Group A}}
\begin{pa}
We will need to know explicitly the exact structure and actions of the group $A_{\bG}(t_{\Sigma})$. To do this we will need to describe explicitly the group $\mathcal{A}$. If $\bG$ is not of type $\D_n$ then $\mathcal{A}$ is cyclic and is described in \cite[Plates I-IX(XII)]{bourbaki:2002:lie-groups-chap-4-6}. To fix the notation in the case of type $\D_n$ we recall the description of $\mathcal{A}$ from \cite[Plate IV(XII)]{bourbaki:2002:lie-groups-chap-4-6}. We denote the elements of $\mathcal{A}$ by the set $\{1,\sigma_1,\sigma_{n-1},\sigma_n\}$\index{sigmai@$\sigma_i$}. If $n \equiv 0 \pmod{2}$ then the element $\sigma_{n-1}$ acts by exchanging the elements in the sets $\{\alpha_0,\alpha_{n-1}\}$, $\{\alpha_1,\alpha_n\}$, $\{\alpha_j,\alpha_{n-j}\}$ for each $2 \leqslant j \leqslant n-2$. The element $\sigma_n$ acts by exchanging $\alpha_j$ with $\alpha_{n-j}$ for all $0 \leqslant j \leqslant n$. Furthermore $\mathcal{A}$ is generated by $\sigma_{n-1}$ and $\sigma_n$. If $n \equiv 1 \pmod{2}$ then the element $\sigma_n$ acts by mapping $\alpha_0 \mapsto \alpha_n \mapsto \alpha_1 \mapsto \alpha_{n-1} \mapsto \alpha_0$ and exchanges $\alpha_j$ with $\alpha_{n-j}$ for $2 \leqslant j \leqslant n-2$. Furthermore $\mathcal{A}$ is generated by $\sigma_n$. The element $\sigma_1$ always acts by exchanging the elements in the sets $\{\alpha_0,\alpha_1\}$, $\{\alpha_{n-1},\alpha_n\}$ and fixes $\alpha_j$ for all $2 \leqslant j \leqslant n-2$.
\end{pa}

\begin{pa}\label{pa:iso-A-Z(G)}
Recall from \cite[3.7]{bonnafe:2005:quasi-isolated} that we have an isomorphism $\mathcal{A} \to \widecheck{\Pi}$. If $\bG$ is simply connected and $\widecheck{\Pi}_{p'} = \widecheck{\Pi}$, (i.e.\ $p$ is a very good prime for $\bG$), then by \cref{lem:centre-fund-grp} we also have an isomorphism $\widecheck{\Pi} \to Z(\bG)$. By composing these isomorphisms we have an isomorphism $\mathcal{A} \to Z(\bG)$. We wish to describe this isomorphism in the case where $\bG$ is of type $\D$. We describe the isomorphism $\mathcal{A} \to \widecheck{\Pi}$ following \cite[\S 3.B]{bonnafe:2005:quasi-isolated}. Let $\alpha_j \in \tilde{\Delta}$ be a root for some $j$ then we denote by $\Delta_j$ the set $\Delta\setminus\{\alpha_j\}$. We write $\Phi_j \subseteq \Phi$ for the parabolic subsystem generated by the set $\Delta_j$ and $\bW_j = \langle s_{\alpha} \mid \alpha\in\Delta_j \rangle$ the corresponding parabolic subgroup of $\bW$. Let $\Phi_j^+ = \Phi_j \cap \Phi^+$ be a system of positive roots for $\Phi_j$ then we denote by $w_j \in \bW_j$ the unique element such that $w_j(\Phi_j^+) = -\Phi_j^+$, (i.e.\ the longest word in $\bW_j$). Define $x_j = w_jw_0 \in \bW$ then $\mathcal{A} = \{x_j \mid m_{\alpha_j} = 1\}$ by \cite[\S 3.5]{bonnafe:2005:quasi-isolated}. The isomorphism $\mathcal{A} \to \widecheck{\Pi}$ is then given by $x_j \mapsto \widecheck{\varpi}_j + \mathbb{Z}\widecheck{\Phi}$.

We consider what this means for a simply connected group of type $\D_n$. In this case we have $\mathcal{A} = \{x_0,x_1,x_{n-1},x_n\}$, (see \cite[Table 1]{bonnafe:2005:quasi-isolated}). It is clear from the description that $x_0$ is the identity. If $j$ is $n-1$ or $n$ then it is easy to determine the action of $x_j$ because the longest word in $\bW_j$ will induce the unique non-trivial graph automorphism on the root system $\Delta_j$ of type $\A_{n-1}$. Furthermore the longest element $w_0 \in \bW$ will induce no graph automorphism if $n$ is even and the unique graph automorphism of order 2 if $n$ is odd. Comparing with \cref{sec:clarification-typeD} we see that we have chosen the labelling such that $\sigma_j \mapsto \hat{z}_j$, (for $j \in \{1,n-1,n\}$), under the isomorphism $\mathcal{A} \to Z(\bG)$.
\end{pa}

\subsection{Component Groups of Semisimple Elements}
\begin{pa}\label{sec:comp-grp-s/s-elmt}
In this section we will be interested in $|A_{\bG^{\star}}(s)|$ where $s \in \bT_0^{\star}$ is a semisimple element such that $\bar{s} = \delta_{\ad}^{\star}(s) \in \bT_{\simc}^{\star}$ is quasi-isolated. An expression for this is already obtained by Bonnaf\'{e} in \cite[Proposition 3.14(b)]{bonnafe:2005:quasi-isolated}, however we wish to determine a slightly different description which will make numerical comparisons between component groups of unipotent elements in $\bG$ simpler. If $\bG^{\star}$ is simply connected then we know that the centraliser of every semisimple element is connected, (by the classical work of Steinberg), so this value will always be 1. It suffices therefore to only consider simple groups which are neither simply connected nor adjoint, (as the adjoint case is dealt with in \cref{tab:quasi-isolated-classical,tab:quasi-isolated-exceptional}). These groups can only occur in types $\A_n$ and $\D_n$ so we may assume that $\bG$ is such a simple group.
\end{pa}

Let $\hat{s} \in \bT_{\ad}^{\star}$ be such that $\delta_{\simc}^{\star}(\hat{s}) = s \in \bT_0^{\star}$. Following Bonnaf\'{e} \cite[\S 2.B]{bonnafe:2005:quasi-isolated} we define two homomorphisms $\omega_{\bar{s}} : C_{\bG_{\simc}^{\star}}(\bar{s}) \to Z(\bG_{\ad}^{\star})$ and $\omega_s : C_{\bG^{\star}}(s) \to Z(\bG_{\ad}^{\star})$\index{omegas@$\omega_s$} by setting $\omega_{\bar{s}}(\bar{x}) = [\hat{s},\hat{x}]$ and $\omega_s(y) = [\hat{s},\hat{y}]$, where $\hat{x}$, $\hat{y}$ are such that $(\delta_{\ad}^{\star}\circ\delta_{\simc}^{\star})(\hat{x}) = \bar{x}$ and $\delta_{\simc}^{\star}(\hat{y}) = y$. We recall the following result of Bonnaf\'{e}.
\begin{lem}[Bonnaf\'{e}, {\cite[Corollary 2.8]{bonnafe:2005:quasi-isolated}}]\label{lem:bonn-centraliser-hom}
The homomorphisms $\omega_{\bar{s}}$, $\omega_s$ induce embeddings $\tilde{\omega}_{\bar{s}} : A_{\bG_{\simc}^{\star}}(\bar{s}) \to Z(\bG_{\ad}^{\star})$ and $\tilde{\omega}_s : A_{\bG^{\star}}(s) \to Z(\bG_{\ad}^{\star})$\index{omegastilde@$\tilde{\omega}_s$}. Their respective images are given by
\begin{align*}
\Image(\tilde{\omega}_{\bar{s}}) &= \{\hat{z} \in Z(\bG_{\ad}^{\star}) \mid \hat{s}\text{ and }\hat{s}\hat{z}\text{ are conjugate in }\bG_{\ad}^{\star}\},\\
\Image(\tilde{\omega}_s) &= \{\hat{z} \in \Ker(\delta_{\simc}^{\star}) \mid \hat{s}\text{ and }\hat{s}\hat{z}\text{ are conjugate in }\bG_{\ad}^{\star}\}.
\end{align*}
\end{lem}
It is easily checked that we have $\omega_{\bar{s}}\circ F_{\bar{s}}^{\star} = F^{\star}\circ\omega_{\bar{s}}$ and $\omega_s\circ F_s^{\star} = F^{\star}\circ\omega_s$. From this lemma we see that $A_{\bG_{\simc}^{\star}}(\bar{s}) \cong \Image(\tilde{\omega}_{\bar{s}})$ and $A_{\bG^{\star}}(s) \cong \Image(\tilde{\omega}_s)$ so to determine $|A_{\bG^{\star}}(s)|$ we need only determine $|\Image(\tilde{\omega}_{\bar{s}}) \cap \Ker(\delta_{\simc}^{\star})|$.

In later sections we will want to compare $|A_{\bG^{\star}}(s)|$ with $|A_{\bG}(u)|$ for some unipotent element $u \in \bG$. We now take the time to prove some small results which will facilitate this.

\begin{lem}\label{lem:Z(G)-ker}
The groups $\Ker(\delta_{\simc})$ and $\Irr(Z(\bG^{\star}))$ are isomorphic and this isomorphism is defined over $\mathbb{F}_q$.
\end{lem}

\begin{proof}
The restriction of the isogeny $\delta_{\simc}$ to the maximal tori $\bT_{\simc} \to \bT_0$ gives rise to an injective homomorphism $\widecheck{X}(\bT_{\simc}) \to \widecheck{X}(\bT_0)$. By \cite[Proposition 1.11]{bonnafe:2006:sln} this induces an isomorphism $(\widecheck{X}(\bT_0)/\widecheck{X}(\bT_{\simc}))_{p'} \cong \Ker(\delta_{\simc})$, where we identify $\widecheck{X}(\bT_{\simc})$ with its image in $\widecheck{X}(\bT_0)$. Using duality this gives rise to an isomorphism
\begin{equation*}
(X(\bT_0^{\star})/X(\bT_{\simc}^{\star}))_{p'} \cong (\widecheck{X}(\bT_0)/\widecheck{X}(\bT_{\simc}))_{p'} \cong \Ker(\delta_{\simc}).
\end{equation*}
Recall that $X(\bT_{\simc}^{\star})$ can be identified with $\mathbb{Z}\Phi^{\star}$ so by \cite[Proposition 4.1]{bonnafe:2006:sln} we have a natural isomorphism $X(Z(\bG)) \cong (X(\bT_0^{\star})/X(\bT_{\simc}^{\star}))_{p'}$. The morphism $X(Z(\bG^{\star})) \to \Irr(Z(\bG^{\star}))$, given by $\chi \mapsto \kappa \circ \chi$ is an isomorphism of finite abelian groups. Finally, checking the statements in \cite{bonnafe:2006:sln}, we can see that all morphisms are defined over $\mathbb{F}_q$.
\end{proof}

\begin{cor}\label{cor:centre-duality}
We have $|\Ker(\delta_{\simc})^F| = |Z(\bG^{\star})^{F^{\star}}|$.
\end{cor}

\begin{proof}
From \cref{lem:Z(G)-ker} we see that $\Ker(\delta_{\simc})^F \cong \Irr(Z(\bG^{\star}))^{F^{\star}}$ because the isomorphism is defined over $\mathbb{F}_q$. The group $\Irr(Z(\bG^{\star}))^{F^{\star}}$ is canonically isomorphic to the group $\Irr(H^1(F^{\star},Z(\bG^{\star})))$, (here $H^1(F^{\star},Z(\bG^{\star}))$ is defined as in \cite[\S1.B]{bonnafe:2006:sln}), because $\chi$ is an element of $\Irr(Z(\bG^{\star}))^{F^{\star}}$ if and only if $\chi(F^{\star}(z)) = \chi(z)$ for all $z \in Z(\bG^{\star})$. On the other hand this is true if and only if $\chi(z^{-1}F^{\star}(z)) = \chi(1)$ for all $z \in Z(\bG^{\star})$, which means $\Ker(\chi) = (F^{\star}-1)Z(\bG^{\star})$. Using, for instance \cite[Exemple 1.1]{bonnafe:2005:quasi-isolated}, we see that
\begin{equation*}
|\Ker(\delta_{\simc})^F| = |\Irr(Z(\bG^{\star}))^{F^{\star}}| = |H^1(F^{\star},Z(\bG^{\star}))| = |Z(\bG^{\star})^{F^{\star}}|.
\end{equation*}
To obtain the second equality we have used the fact that $H^1(F^{\star},Z(\bG^{\star}))$ has the same order as its character group because it is a finite abelian group.
\end{proof}

We finally end this discussion on component groups with a particularly useful observation relating fixed point groups.

\begin{lem}\label{lem:comp-group-fixed-point}
Assume $\bG$ has a cyclic centre and $A \leqslant \Ker(\delta_{\simc}^{\star})$ and $Z \leqslant Z(\bG)$ are subgroups of common order $d = |A| = |Z|$ then $|A^{F^{\star}}| = |Z^F|$.
\end{lem}

\begin{proof}
By \cref{lem:Z(G)-ker} we have $|\Ker(\delta_{\simc}^{\star})| = |Z(\bG)| = N$, hence these groups are isomorphic to a cyclic group of order $N$. We may assume that $1 \leqslant n,m\leqslant N$ are such that $F^{\star}(x) = x^n$ and $F(y) = y^m$, where $\Ker(\delta_{\simc}^{\star}) = \langle x \rangle$ and $Z(\bG) = \langle y \rangle$. Taking $\lambda = N/d$ it is easily seen that $|A^{F^{\star}}| = \lambda\cdot\gcd(n-1,d)$ and $|L^F| = \lambda\cdot\gcd(m-1,d)$. To show that these groups have the same order it is enough to show that $\gcd(n-1,d) = \gcd(m-1,d)$. However, because $\gcd(n-1,d) = \gcd(n-1,d,N) = \gcd(n-1,N)$, (similarly for $\gcd(m-1,d)$), it is sufficient to show that $\gcd(n-1,N) = \gcd(m-1,N)$ but this is just a restatement of \cref{cor:centre-duality}.
\end{proof}

\subsection{\texorpdfstring{Groups of Type $\A_n$}{Groups of Type A}}
\begin{pa}\label{sec:comp-orders-type-A}
If $\bG$ is a group of type $\A_n$ then $Z(\bG_{\ad}^{\star})$ is a cyclic group so this simplifies trying to understand $|A_{\bG^{\star}}(s)|$. We know $|A_{\bG_{\simc}^{\star}}(\bar{s})| = |\Image(\tilde{\omega}_{\bar{s}})|$ in particular, as $\Ker(\delta_{\simc}^{\star})$ is cyclic, we know $z \in \Image(\tilde{\omega}_{\bar{s}}) \cap \Ker(\delta_{\simc}^{\star})$ if and only if the order of $z$ divides $|\Image(\tilde{\omega}_{\bar{s}})|$ and $|\Ker(\delta_{\simc}^{\star})|$. Hence it is easy to see that we have
\begin{equation*}
|A_{\bG^{\star}}(s)| = \gcd(|A_{\bG_{\simc}^{\star}}(\bar{s})|,|\Ker(\delta_{\simc}^{\star})|).
\end{equation*}
Assume $u \in \bG$ is a unipotent element and $u_{\simc}$ is the unique unipotent element in $\delta_{\simc}^{-1}(u)$. Using the natural exact sequence $\Ker(\delta_{\simc}) \to A_{\bG_{\simc}}(u_{\simc}) \to A_{\bG}(u) \to 1$ we have
\begin{equation*}
|A_{\bG}(u)| = |A_{\bG_{\simc}}(u_{\simc})|/\gcd(|\Ker(\delta_{\simc})|,|A_{\bG_{\simc}}(u_{\simc})|).
\end{equation*}
Let $d := |A_{\bG}(u)|$, which is a divisor of $|A_{\bG_{\simc}}(u_{\simc})|$. By the description of $|A_{\bG_{\simc}}(u_{\simc})|$ given in \cite[\S10.3]{lusztig:1984:intersection-cohomology-complexes} we have $d$ is a divisor of $n+1$ and $p \nmid d$. Therefore there exists a semisimple element $s \in \bT_0^{\star}$ such that $\bar{s} = \delta_{\ad}^{\star}(s)$ is quasi-isolated and $|A_{\bG_{\simc}^{\star}}(\bar{s})| = d$. As $A_{\bG}(u) = Z_{\bG}(u)$ we have $d$ divides $|Z(\bG)| = |\Ker(\delta_{\simc}^{\star})|$ so $|A_{\bG^{\star}}(s)| = d$. What we have shown here is that for each unipotent conjugacy class $\mathcal{O}$ of $\bG$ there exists a semisimple element $s \in \bT_0^{\star}$ such that $|A_{\bG}(u)| = |A_{\bG^{\star}}(s)|$.
\end{pa}

\subsection{\texorpdfstring{Groups of Type $\D_n$}{Groups of Type D}}
\begin{pa}\label{subsec:groups-type-Dn}
If $\bG$ is a simple group of type $\D_n$, which is neither simply connected nor adjoint, such that $n \equiv 1 \pmod{2}$ then $\bG$ must be a special orthogonal group. The kernel $\Ker(\delta_{\simc}^{\star})$ will be the unique subgroup of order 2 in $Z(\bG_{\simc})$ so
\begin{equation*}
|A_{\bG^{\star}}(s)| = \begin{cases}
1 &\text{if }|A_{\bG_{\ad}}(\bar{s})| = 1,\\
2 &\text{if }|A_{\bG_{\ad}}(\bar{s})| \geqslant 2.
\end{cases}
\end{equation*}
If $n \equiv 0 \pmod{2}$ then $\bG$ is either isomorphic to a special orthogonal group or a half-spin group. It is clear that if $|A_{\bG_{\ad}}(\bar{s})| = 1$ or 4 then we will respectively have $|A_{\bG^{\star}}(s)| = 1$ or 2. The problem now arises when $|A_{\bG_{\ad}}(\bar{s})| = 2$. Assume $\bar{s}$ is a quasi-isolated semisimple element with this property then in \cref{tab:quasi-isolated-Dn} we describe the orders of $|A_{\bG^{\star}}(s)|$ depending upon whether $\bG$ is a special orthogonal group or a half-spin group. To determine the information in \cref{tab:quasi-isolated-Dn} one has to only check which element of $\mathcal{A}$ stabilises $\Sigma$ then see if the corresponding element of $Z(\bG_{\simc})$ lies in $\Ker(\delta_{\simc}^{\star})$, (using the description in \cref{pa:iso-A-Z(G)}). Please Note: from the discussion in \cref{sec:clarification-typeD} we always have $\Ker(\delta_{\simc}^{\star}) = \langle \hat{z}_n \rangle$ when $\bG$ is a half-spin group, chosen as in \cref{sec:clarification-typeD}. Hence the information in \cref{tab:quasi-isolated-Dn} holds regardless of the congruence of $n\pmod{4}$.

\begin{table}[h!t]
\centering
\begin{tabular}{>{$}l<{$}>{$}c<{$}>{$}c<{$}>{$}c<{$}>{$}c<{$}}
\toprule
\multicolumn{1}{c}{$\Sigma$} & C_{\bG_{\ad}^{\star}}(\bar{s})^{\circ} & \SO_{2n}(\mathbb{K}) & \HSpin_{2n}(\mathbb{K}) \tabularnewline\midrule
\{\alpha_d\}, 2 \leqslant d < n/2
&
\D_d\D_{n-d}
&
2
&
1
\\
\{\alpha_{n/2}\}\text{, (only if }2\mid n)
&
\D_{n/2}\D_{n/2}
&
2
&
2
\\
\{\alpha_0,\alpha_1\}
&
\D_{n-1}
&
2
&
1
\\
\{\alpha_0,\alpha_{n-1}\}
&
\A_{n-1}
&
1
&
1
\\
\{\alpha_0,\alpha_n\}
&
\A_{n-1}
&
1
&
2
\\
\bottomrule
\end{tabular}
\caption{Component Group Orders in Groups of Type $\D_n$}
\label{tab:quasi-isolated-Dn}
\end{table}
\end{pa}

\subsection{\texorpdfstring{$F$-stability of Classes in Adjoint Groups}{Stability of Classes in Adjoint Groups}}
In the remainder of this section we wish to address two issues. Firstly we wish to show that, modulo some exceptions, every class of quasi-isolated semisimple elements in a simple adjoint algebraic group is $F$-stable. Secondly we wish to show that, if $\mathcal{C}$ is an $F$-stable class of quasi-isolated semisimple elements of $\bG_{\ad}$ then there exists an $F$-stable class of $\bG_{\simc}$ whose image under $\delta_{\ad}\circ\delta_{\simc}$ is $\mathcal{C}$.

\begin{prop}\label{prop:F-stable-classical-quasi-adjoint}
Let $\bG_{\ad}$ be a simple adjoint algebraic group of classical type and $F$ a Frobenius endomorphism written as $F_r\circ\tau$ where $\tau$ is a graph automorphism of $\bG_{\ad}$. Given any set $\Sigma \in \mathcal{Q}(\bG_{\ad})$ we have $t_{\Sigma}$ is conjugate to $F(t_{\Sigma})$ unless:
\begin{itemize}
	\item $\bG_{\ad}$ is of type $\D_n$, the graph automorphism $\tau$ is of order 2 and $\Sigma$ is $\{\alpha_0,\alpha_{n-1}\}$ or $\{\alpha_0,\alpha_n\}$.
	\item $\bG_{\ad}$ is of type $\D_4$, the graph automorphism $\tau$ is of order 3 and $\Sigma = \{\alpha_0,\alpha_1\}$, $\{\alpha_0,\alpha_3\}$ or $\{\alpha_0,\alpha_4\}$.
\end{itemize}
In particular, except for those mentioned above, the conjugacy class containing $t_{\Sigma}$ is $F$-stable.
\end{prop}

\begin{proof}
Let $\mathcal{C}_{\Sigma}$ be the class of quasi-isolated semisimple elements of $\bG_{\ad}$ such that $t_{\Sigma} \in \mathcal{C}_{\Sigma}$. Furthermore let $y_{\Sigma} = \imath_{\bT_{\ad}}^{-1}(t_{\Sigma}) = \sum_{\alpha  \in \Sigma} 1/m_{\alpha}|\Sigma|\otimes\widecheck{\varpi}_{\alpha} \in (\mathbb{Q/Z})_{p'} \otimes \widecheck{X}(\bT_{\ad})$ and recall that $m_{\alpha}$ is constant on $\Sigma$. Using the fact that the tensor product is taken over $\mathbb{Z}$ we have the action of $F$ on $y_{\Sigma}$ is given by $F(y_{\Sigma}) = \sum_{\alpha  \in \Sigma} r/m_{\alpha}|\Sigma|\otimes\widecheck{\varpi}_{\tau(\alpha)}$. To show $\mathcal{C}_{\Sigma}$ is $F$-stable we need only show that $y_{\Sigma}$ and $F(y_{\Sigma})$ lie in the same $\bW$-orbit.
\begin{itemize}
	\item If $\Sigma = \{\alpha_0\}$ then $y_{\Sigma} = 0$, which is always $F$-stable.
	
	\item Let $\bG_{\ad}$ be of type $\A_n$ and assume $\tau$ is trivial. In this instance it will be much more transparent to work with a concrete realisation of $\bG_{\ad}$, namely $\PGL_{n+1}(\mathbb{K})$. Let $d$ be a divisor of $n+1$ and $\Sigma$ the corresponding subset of the roots. Following Bonnaf\'{e} we define a matrix $J_d = \diag(1,\eta_d,\eta_d^2,\dots,\eta_d^{d-1}) \in \GL_d(\mathbb{K})$, where $\eta_d$ is a primitive $d$th root of unity in $\mathbb{K}$. Let $\tilde{s}_{\Sigma} = I_{n+1/d} \otimes J_d \in \GL_{n+1}(\mathbb{K})$ be the Kronecker product of the matrices, where $I_{n+1/d} \in\GL_{n+1/d}(\mathbb{K})$ is the identity matrix. Considering the standard quotient map $\pi : \GL_{n+1}(\mathbb{K}) \to \PGL_{n+1}(\mathbb{K})$ we have $s_{\Sigma} = \pi(\tilde{s}_{\Sigma})$ is a representative in $\PGL_{n+1}(\mathbb{K})$ of the class parameterised by $\Sigma$. The action of the Frobenius is given by $F(\tilde{s}_{\Sigma}) = I_{n+1/d} \otimes \diag(1,\eta_d^q,\eta_d^{2q},\dots,\eta_d^{q(d-1)})$. As $q$ and $d$ are coprime we have $\eta_d^{q} = \eta_d^i$ for some $1 \leqslant i \leqslant d-1$ so the entries $\eta_d^q,\dots,\eta_d^{q(d-1)}$ are just a permutation of $\eta_d,\dots,\eta_d^{d-1}$. There is clearly an element of the Weyl group $w_d \in \bW$ such that $F(\tilde{s}_{\Sigma}) = \tilde{s}_{\Sigma}^{w_d}$. If $\tilde{\mathcal{C}}_{\Sigma}$ is the conjugacy class of $\GL_{n+1}(\mathbb{K})$ containing $\tilde{s}_{\Sigma}$ then $\pi(\tilde{\mathcal{C}}_{\Sigma}) = \mathcal{C}_{\Sigma}$ and $F(\tilde{\mathcal{C}}_{\Sigma}) = \mathcal{\tilde{C}}_{\Sigma}$. As $\pi$ is defined over $\mathbb{F}_q$ we have $\pi(F(\tilde{\mathcal{C}}_{\Sigma})) = \pi(\tilde{\mathcal{C}}_{\Sigma}) \Rightarrow F(\pi(\tilde{\mathcal{C}}_{\Sigma})) = F(\mathcal{C}_{\Sigma}) = \mathcal{C}_{\Sigma}$ so $\mathcal{C}_{\Sigma}$ is an $F$-stable class.
	
	\item Let $\bG_{\ad}$ be of type $\A_n$ and assume $\tau$ is of order 2. The map $\tau$ acts on the simple roots by sending $\alpha_k \mapsto \alpha_{n+1-k}$ for all $1 \leqslant k \leqslant n$. Furthermore it is such that $\tau(\alpha_0) = \alpha_0$. The roots in $\Sigma$ are all of the form $\alpha_{j(n+1)/d}$ for some $0 \leqslant j \leqslant d-1$, where $d$ is a divisor of $n+1$ as in the previous case. If $j \neq 0$ then we have $\tau(\alpha_{j(n+1)/d}) = \alpha_{(d-j)(n+1)/d}$ so it is clear that $\tau$ preserves the set $\Sigma$. Hence $y_{\Sigma}$ and $F(y_{\Sigma})$ are in the same $\bW$-orbit.
	
	\item Let $\bG_{\ad}$ be of type $\B_n$, $\C_n$ or $\D_n$ and assume $\tau$ is trivial. If $m_{\alpha}|\Sigma| = 2$ then $F(y_{\Sigma}) = y_{\Sigma}$ because $q$ is odd hence $q/2 = 1/2 \in \mathbb{Q/Z}$.
	
	Assume $m_{\alpha}|\Sigma| =4$ then $\bG_{\ad}$ must be of type $\C_n$ or $\D_n$. If $q \equiv 1 \pmod{4}$ then $q/4 = 1/4 \in \mathbb{Q/Z}$ and $F(y_{\Sigma}) = y_{\Sigma}$. If $q \equiv 3 \pmod{4}$ then $q/4 = 3/4 = -1/4 \in \mathbb{Q/Z}$ so $F(y_{\Sigma}) = -y_{\Sigma}$. If $\bG_{\ad}$ is of type $\C_n$ or it is of type $\D_n$ and $n \equiv 0 \pmod{2}$ then the longest element $w_0 \in \bW$ acts on the coweights by $-1$, (see \cite[Plates II - IV (XI)]{bourbaki:2002:lie-groups-chap-4-6}), therefore $F(y_{\Sigma})$ and $y_{\Sigma}$ are conjugate by $w_0$. If $\bG_{\ad}$ is of type $\D_n$ and $n \equiv 1\pmod{2}$ then the longest element $w_0 \in \bW$ acts on the coweights as $-\epsilon$ where $\epsilon$ is such that $\varepsilon(\widecheck{\varpi}_i) = \widecheck{\varpi}_i$ for all $1 \leqslant i \leqslant n-2$ and $\varepsilon$ exchanges $\widecheck{\varpi}_{n-1}$ and $\widecheck{\varpi}_n$, (see \cite[Plate IV (XI)]{bourbaki:2002:lie-groups-chap-4-6}). All subsets $\Sigma$ considered here are stable under $\epsilon$ so $F(y_{\Sigma})$ and $y_{\Sigma}$ are conjugate by $w_0$.
	
	\item Let $\bG_{\ad}$ be of type $\D_n$ and assume $\tau$ is of order 2. It is clear that all subsets $\Sigma$ are stable under $\tau$, except for $\Sigma = \{\alpha_0,\alpha_n\}$ or $\{\alpha_0,\alpha_{n-1}\}$. As $\tau$ cannot be induced by an element of $\bW$ we cannot have $F(t_{\Sigma})$ conjugate to $t_{\Sigma}$ in these two cases.
	
	\item Let $\bG_{\ad}$ be of type $\D_4$ and assume $\tau$ is of order 3. The possible subsets $\Sigma \subset \tilde{\Delta}$ are $\{\alpha_0\}$, $\{\alpha_2\}$, $\{\alpha_0,\alpha_1,\alpha_3,\alpha_4\}$, $\{\alpha_0,\alpha_1\}$, $\{\alpha_0,\alpha_3\}$ and $\{\alpha_0,\alpha_4\}$. The first three sets are the only ones stable under $\tau$. As $\tau$ cannot be induced by an element of $\bW$ we cannot have $F(t_{\Sigma})$ conjugate to $t_{\Sigma}$ in these remaining cases.
\end{itemize}
\end{proof}

If $\bG_{\ad}$ is simple and of exceptional type then we can compute the $F$-stability of the semisimple classes using the development version of CHEVIE maintained by Jean Michel at \cite{michel:2011:CHEVIE-development}. This uses the implementation of semisimple elements in \cite{geck-hiss:1996:CHEVIE}, which was done by Bonnaf\'{e} and Michel as a part of their computational proof of the Mackey formula, (see \cite{bonnafe-michel:2011:mackey-formula}). In \cite[Appendix D]{taylor:2012:thesis} we give the code for a program which will verify the $F$-stability of any class of quasi-isolated semisimple elements. Using this program we get the following proposition.

\begin{prop}
Let $\bG_{\ad}$ be a simple adjoint algebraic group of exceptional type then any class of quasi-isolated semisimple elements in $\bG_{\ad}$ is $F$-stable.
\end{prop}

\subsection{\texorpdfstring{$F$-stability of Classes in Reductive Groups}{Stability of Classes in Reductive Groups}}
\begin{pa}
Assume $\bH$ is a connected reductive algebraic group with simple derived subgroup $\bH'$ and let $\delta_{\ad} : \bH \to \bH_{\ad}$ be an adjoint quotient which is defined over $\mathbb{F}_q$. We would like to positively answer the following question: Given an $F$-stable class $\mathcal{C}_{\ad}$ of quasi-isolated semisimple elements in $\bH_{\ad}$ does there exist an $F$-stable class $\mathcal{C}$ in $\bH$ such that $\delta_{\ad}(\mathcal{C}) = \mathcal{C}_{\ad}$?

Assume $\delta_{\ad}' : \bH' \to \bH_{\ad}$ is an adjoint quotient of $\bH'$. If there exists an $F$-stable class $\mathcal{C}'$ of $\bH'$ such that $\delta_{\ad}'(\mathcal{C}') = \mathcal{C}_{\ad}$, then such a class would be a solution to our question. This is clear as $\bH$ is an almost direct product of $\bH'$ and $Z(\bH)$. Let $\delta_{\simc}' : \bH_{\simc} \to \bH'$ be a simply connected cover of $\bH'$ and assume that $\mathcal{C}_{\simc}$ is an $F$-stable class of semisimple elements such that $(\delta_{\ad}'\circ\delta_{\simc}')(\mathcal{C}_{\simc}) = \mathcal{C}_{\ad}$ then $\mathcal{C}' = \delta_{\simc}'(\mathcal{C}_{\simc})$ is a class as required. Following this we may assume without loss of generality that $\bH$ is a simple simply connected algebraic group.
\end{pa}

\begin{prop}\label{prop:F-stable-surjective-class-semisimple}
Assume $\bG$ is simple and simply connected. Let $\mathcal{C}_{\ad}$ be an $F$-stable class of quasi-isolated semisimple elements then there exists an $F$-stable class of semisimple elements $\mathcal{C}$ such that $\delta_{\ad}(\mathcal{C}) = \mathcal{C}_{\ad}$.
\end{prop}

\begin{proof}
We keep the notational conventions specified in the proof of \cref{prop:F-stable-classical-quasi-adjoint}.
\begin{itemize}
	\item Assume $\mathcal{C}_{\ad} = \{1\}$ then clearly $\mathcal{C} = \{1\}$ satisfies our conditions.
	
	\item We will deal with many cases by using the following argument inspired by Lusztig and Bonnaf\'{e}, (see \cite[\S 8]{lusztig:1988:reductive-groups-with-a-disconnected-centre} and \cite[\S 2.B]{bonnafe:2005:quasi-isolated}). Let $s \in \mathcal{C}_{\ad}^F$ and choose $\hat{s} \in \bG$ such that $\delta_{\ad}(\hat{s}) = s$. If we can show $\hat{s}$ is conjugate to $F(\hat{s})$ then the class containing $\hat{s}$ is $F$-stable and its image under $\delta_{\ad}$ is $\mathcal{C}_{\ad}$. Note that we have $\delta_{\ad}(F(\hat{s})) = s$ so we must have $F(\hat{s}) = \hat{s}\hat{z}$ for some $\hat{z} \in Z(\bG)$. By \cref{lem:bonn-centraliser-hom} we know $A_{\bG_{\ad}}(s)$ is isomorphic to the group $\{\hat{z} \in Z(\bG) \mid \hat{s}\hat{z}$ is conjugate to $\hat{s}$ in $\bG\}$, hence whenever $|A_{\bG_{\ad}}(s)| = |Z(\bG)|$ we are done.
	
	\item Let $\bG_{\ad}$ be $\PGL_{n+1}(\mathbb{K})$. Consider the matrix $\tilde{s}_{\Sigma} = I_{(n+1)/d} \otimes J_d \in \GL_{n+1}(\mathbb{K})$ specified in the proof of \cref{prop:F-stable-classical-quasi-adjoint}; then we already know that this lies in an $F$-stable conjugacy class of $\GL_{n+1}(\mathbb{K})$. We have $\det(\tilde{s}_{\Sigma}) = \pm 1$ so it is clear that $\hat{s}_{\Sigma} = \det(\tilde{s}_{\Sigma})\tilde{s}_{\Sigma} \in \SL_{n+1}(\mathbb{K})$, which is a simply connected group of type $\A_n$. Applying the Frobenius we see that $F(\hat{s}_{\Sigma}) = \det(\tilde{s}_{\Sigma})F(\tilde{s}_{\Sigma})$ but we know $F(\tilde{s}_{\Sigma})$ and $\tilde{s}_{\Sigma}$ are conjugate by an element of $\bW$ so clearly $F(\hat{s}_{\Sigma})$ and $\hat{s}_{\Sigma}$ are conjugate. Therefore we can take $\mathcal{C}$ to be the conjugacy class containing $\hat{s}_{\Sigma}$.
\end{itemize}

To prove the remaining cases we will consider the following argument. Recall that we have $\widecheck{X}(\bT_0) \subseteq \widecheck{X}(\bT_{\ad})$ and $\widecheck{X}(\bT_0)$ can be identified with the coroot lattice $\mathbb{Z}\widecheck{\Phi}$. Consider the element $y_{\Sigma} \in (\mathbb{Q/Z})_{p'}\otimes\widecheck{X}(\bT_{\ad})$ then we can express this as a sum of the simple coroots. We then have an element $\hat{y}_{\Sigma} = \sum_{\alpha \in \Delta} r_{\alpha}\otimes\widecheck{\alpha}$ for some $r_{\alpha} \in (\mathbb{Q/Z})_{p'}$, which gives us a representative $\imath_{\bT_0}(\hat{y}_{\Sigma})$ of the preimage $\delta_{\ad}^{-1}(\imath_{\bT_{\ad}}(y_{\Sigma}))$. We now argue that $\hat{y}_{\Sigma}$ is conjugate to $F(\hat{y}_{\Sigma})$ so we can take $\mathcal{C}$ to be the class containing $\hat{y}_{\Sigma}$.

\begin{itemize}
	\item Let $\bG_{\ad}$ be of type $\B_n$, $\C_n$ or $\D_n$ then we need only consider the case where $m_{\alpha}|\Sigma| = 2$. Using the information in \cite[Plate II-IV(VI)]{bourbaki:2002:lie-groups-chap-4-6} we see that $r_{\alpha}$ is either 1, $\frac{1}{2}$, $\frac{1}{4}$ or $\frac{3}{4}$ for all $\alpha \in \Delta$. Furthermore in the case of type $\D_{n-1}$ we have $r_{\alpha_{n-1}} = r_{\alpha_n}$ if $\Sigma \neq \{\alpha_0,\alpha_{n-1}\}$ or $\{\alpha_0,\alpha_n\}$. By the arguments used in the proof of \cref{prop:F-stable-classical-quasi-adjoint} it is then clear that $\hat{y}_{\Sigma}$ and $F(\hat{y}_{\Sigma})$ are conjugate, (note that we also have $w_0$ acts as $-1$ in the case of type $\B_n$).
	
	\item Let $\bG_{\ad}$ be of type $\E_6$. We need only consider the case where $\Sigma = \{\alpha_2\}$. Using the information in \cite[Plate V(VI)]{bourbaki:2002:lie-groups-chap-4-6} we see
	\begin{equation*}
	\hat{y}_{\{\alpha_2\}} = \frac{1}{2}\otimes\widecheck{\alpha}_1 + \frac{1}{2}\otimes\widecheck{\alpha}_4 + \frac{1}{2}\otimes\widecheck{\alpha}_6.
	\end{equation*}
	This is clearly fixed by $F$.
	
	\item Let $\bG_{\ad}$ be of type $\E_7$. We need only consider the cases where $\Sigma = \{\alpha_1\}$ or $\{\alpha_3\}$. Using the information in \cite[Plate VI(VI)]{bourbaki:2002:lie-groups-chap-4-6} we see that
	\begin{align*}
	\hat{y}_{\{\alpha_1\}} &= \frac{1}{2}\otimes\widecheck{\alpha}_3 + \frac{1}{2}\otimes\widecheck{\alpha}_5 + \frac{1}{2}\otimes\widecheck{\alpha}_7,\\
	\hat{y}_{\{\alpha_3\}} &= \frac{1}{3}\otimes\widecheck{\alpha}_2 + \frac{2}{3}\otimes\widecheck{\alpha}_4 + \frac{1}{3}\otimes\widecheck{\alpha}_6 + \frac{2}{3}\otimes\widecheck{\alpha}_7.
	\end{align*}
	It is clear that the element $\hat{y}_{\{\alpha_1\}}$ is $F$-fixed and if $q \equiv 1 \pmod{3}$ then $\hat{y}_{\{\alpha_3\}}$ is $F$-fixed. If $q \equiv 2 \pmod{3}$ then $F(\hat{y}_{\{\alpha_3\}}) = -\hat{y}_{\{\alpha_3\}}$ but by \cite[Plates VI(XI)]{bourbaki:2002:lie-groups-chap-4-6} we see that the longest element $w_0 \in \bW$ acts on the coroots by $-1$.
\end{itemize}
\end{proof}

If $\bG$ is of type $\B_n$, $\C_n$ or $\D_n$ it will be useful to know that we can choose class representatives with a specific action of $F$. In particular we record the following corollary for later, which follows immediately from the proofs of \cref{prop:F-stable-classical-quasi-adjoint,prop:F-stable-surjective-class-semisimple}.

\begin{cor}\label{cor:quasi-isolated-action}
Assume $\bG$ is simply connected of type $\B_n$, $\C_n$ or $\D_n$ and $\Sigma \in \mathcal{Q}(\bG_{\ad})$ is such that $m_{\alpha}|\Sigma| = 2$. Then there exists an element $s \in \bT_0$ such that $\delta_{\ad}(s) = t_{\Sigma}$ and either $F(s) = s$ or $F(s) = s^{w_0}$, (where $w_0\in\bW$ is the longest element).
\end{cor}

\subsection{Remarks on Fixing Class Representatives}
\begin{pa}\label{sec:fixing-semisimple-class-reps}
We assume that $\bG$ is a simple algebraic group. We will now fix representatives, which we assume chosen once and for all, for conjugacy classes surjecting onto quasi-isolated classes. Let $\Sigma \in \mathcal{Q}(\bG_{\simc}^{\star})$ be a set of roots and let $t_{\Sigma} \in \bT_{\simc}^{\star}$ be its corresponding semisimple element. We assume that the class containing $t_{\Sigma}$ is $F^{\star}$-stable. Let $\btG'$ be the derived subgroup of $\btG^{\star}$ and denote by $\delta_{\simc}' : \bG_{\ad}^{\star} \to \btG'$ and $\delta_{\ad}' : \btG' \to \bG_{\simc}^{\star}$ a simply connected cover and adjoint quotient of $\btG'$. By \cref{prop:F-stable-surjective-class-semisimple} we may fix a semisimple element $s_{\simc} \in \bT_{\ad}^{\star}$ such that the conjugacy class of $\bG_{\ad}^{\star}$ containing $s_{\simc}$ is $F^{\star}$-stable and $(\delta_{\ad}'\circ\delta_{\simc}')(s_{\simc}) = t_{\Sigma}$, (this also satisfies the conditions of \cref{cor:quasi-isolated-action}). We then define the semisimple element $\tilde{s} \in \btG^{\star}$ to be the image of $\delta_{\simc}'(s_{\simc})$ under the natural embedding $\btG' \to \btG^{\star}$.

Note that the surjective morphism $\iota^{\star} : \btG^{\star} \to \bG^{\star}$ restricts to an isogeny $\btG' \to \bG^{\star}$ defined over $\mathbb{F}_q$. We define $s \in \bT_0^{\star}$ to be the image of $\tilde{s}$ under $\iota^{\star}$. This element lies in an $F^{\star}$-stable conjugacy class of $\bG^{\star}$ and the image $\delta_{\ad}^{\star}(s)$ in $\bG_{\simc}^{\star}$ will be $t_{\Sigma}$. By defining the representatives in this way we have $s$, $\tilde{s}$ and $t_{\Sigma}$ are all conjugate to their image under the Frobenius by the same element of the Weyl group, (where here we identify the Weyl groups through the appropriate morphisms). In particular we will be able to uniformly describe the automorphism induced by the Frobenius endomorphism on the component groups of their centralisers.
\end{pa}

\section{\texorpdfstring{General Strategy for the Proof of \cref{prop:A}}{General Strategy}}\label{sec:gen-strategy}
\begin{pa}
In the following sections we will carry out the case by case check giving the proof of \cref{prop:A}. In each section we will state results of Lusztig, (from \cite{lusztig:2009:unipotent-classes-and-special-Weyl}), or H\'{e}zard, (from \cite{hezard:2004:thesis}), which provide for each $\mathcal{O} \in \Clu(\bG)^F$ the following information:
\begin{itemize}
	\item a parahoric subgroup $\bW' \leqslant \bW$, (i.e.\ a subgroup conjugate to $\langle \mathbb{I} \rangle$ for some $\mathbb{I} \subseteq \mathbb{S}_0$),
	\item a special character $\chi \in \Irr(\bW')$ such that $j_{\bW'}^{\bW}(\chi) = \rho(\mathcal{O})$ and $f_{\chi} = |A_{\btG}(u)|$.
\end{itemize}
Here $f_{\chi}$\index{fchi@$f_{\chi}$} is the value defined in \cite[1.1]{lusztig:2009:unipotent-classes-and-special-Weyl} in terms of the generic degree polynomial of the generic Hecke algebra representation corresponding to $\chi$. Actually in \cite{hezard:2004:thesis} H\'{e}zard phrases things in terms of the order of a small finite group attached to the family containing $\chi$ but one easily relates this to $f_{\chi}$ using \cite[4.14.2]{lusztig:1984:characters-of-reductive-groups}. Using the results from \cref{sec:quasi-isolated} we will show that there exists an element $\tilde{s} \in \btT_0^{\star}$ such that $\bW^{\star}(\tilde{s})$ is identified with $\bW'$ under the anti-isomorphism $\bW \to \bW^{\star}$. Furthermore we will show that $\chi$ is stable under $F_{\tilde{s}}^{\star}$ so that $(\tilde{s},\bW^{\star}(\tilde{\mathcal{F}})) \in \mathcal{T}_{\btG}$.
\end{pa}

\begin{rem}
Note that in \cite{lusztig:2009:unipotent-classes-and-special-Weyl} Lusztig works with algebraic groups defined over $\mathbb{C}$. However Lusztig remarks in \cite[1.8]{lusztig:2009:unipotent-classes-and-special-Weyl} that his results also hold for fields of very good characteristic. In fact, after replacing the description of $\tilde{\mathbf{z}}_{C}$ in \cite[3.2]{lusztig:2009:unipotent-classes-and-special-Weyl} with that given in \cite[10.3]{lusztig:1984:intersection-cohomology-complexes} one easily sees that these arguments work in type $\A_n$ in all characteristics, (see also \cite[\S5.1]{taylor:2012:thesis}). Hence Lusztig's results hold in good characteristic.
\end{rem}

\begin{pa}
Let $\bW^{\star}(\tilde{s}) = \bW_1^{\star} \times \cdots \times \bW_r^{\star}$ be a decomposition into irreducible Weyl groups, then we have a corresponding decomposition of the family $\bW^{\star}(\mathcal{\tilde{F}}) = \bW^{\star}(\tilde{\mathcal{F}}_1) \boxtimes \cdots \boxtimes \bW^{\star}(\mathcal{\tilde{F}}_r)$ determined by $\chi = \chi_1\boxtimes\cdots\boxtimes\chi_r$. The family $\mathcal{F}$ of unipotent characters of $C_{G^{\star}}(s)^{\circ}$ in bijective correspondence with the family $\tilde{\mathcal{F}}$ will then decompose as a corresponding product of families $\mathcal{F} = \mathcal{F}_1 \boxtimes \cdots \boxtimes \mathcal{F}_r$. We choose unipotent characters $\psi_i \in \mathcal{F}_i$ such that $n_{\psi_i} = f_{\chi_i}$ for each $1 \leqslant i \leqslant r$, (note that such characters always exist by \cite[4.26.3]{lusztig:1984:characters-of-reductive-groups}). We also make the extra assumption that if we have an isomorphism $\varphi : \bW_i^{\star} \to \bW_j^{\star}$ for some $i \neq j$ such that $\chi_i = \chi_j\circ\varphi$ then we take $\psi_i = \psi_j$, (i.e.\ the unipotent characters are parameterised in the same way). Take $\psi$ to be the product $\psi_1 \boxtimes \cdots \boxtimes \psi_r$ then our goal now is to show that the triple $(\tilde{s},\bW^{\star}(\tilde{\mathcal{F}}),\psi)$ satisfies \crefrange{P1}{P5}. Firstly as $n_{\psi} = n_{\psi_1}\cdots n_{\psi_r} = f_{\chi_1}\cdots f_{\chi_r} = f_{\chi} = |A_{\btG}(u)|$ we have \cref{P1} holds by the results of Lusztig and H\'{e}zard. Clearly \cref{P3} holds so this leaves us with \cref{P2,P4,P5}.
\end{pa}

\begin{pa}
Assume $\bG$ is of type $\B_n$, $\C_n$ or $\D_n$ and $\mathcal{O} \in \Clu(\bG)^F$ is such that $A_{\bG}(u)$ is abelian. From the details given in the following sections it will be clear that the chosen pair $(\tilde{s},\bW^{\star}(\tilde{\mathcal{F}}))$ satisfies the following condition: the structure of the centraliser $C_{\btG}(\tilde{s})$ is determined by a set $\Sigma$ satisfying $m_{\alpha}|\Sigma| = 2$ (see \cref{tab:quasi-isolated-classical}) -- or in other words $\tilde{s}$ is an involution. When $\tilde{s}$ is isolated then the validity of \cref{P4} follows from \cite[Proposition 2.3]{geck-hezard:2008:unipotent-support} because \cref{P1,P3} hold. However checking the details in \cite[Chapter 3]{hezard:2004:thesis} one sees that a slightly stronger result holds, namely that this is true whenever $\tilde{s}$ is an involution. We will see that under these assumptions $\tilde{s}$ is chosen to be an involution, hence we will only need to check \cref{P4} in type $\A_n$ and the exceptional types. This means that we will mainly concern ourselves with \cref{P2,P5}.
\end{pa}

\begin{pa}
To make stating the results of Lusztig slightly more convenient we adopt the following convention, which we maintain until the end. Consider a subset $\mathbb{I} \subseteq \mathbb{S}_0$ then we have a corresponding parahoric subgroup $\bW_{\mathbb{I}} \leqslant \bW$ generated by $\mathbb{I}$. Let $\rho \in \Irr(\bW_{\mathbb{I}})$ be an irreducible character of $\bW_{\mathbb{I}}$ and let $\mathcal{A}_{\mathbb{I}} \leqslant \mathcal{A}$ be the subgroup of elements normalising $\bW_{\mathbb{I}}$. We will write $\Stab_{\mathcal{A}}(\rho)$\index{StabA@$\Stab_{\mathcal{A}}(\rho)$} to denote the subgroup of $\mathcal{A}_{\mathbb{I}}$ which stabilises $\rho$ under the natural conjugation action of $\mathcal{A}_{\mathbb{I}}$ on $\bW_{\mathbb{I}}$.
\end{pa}

\section{\texorpdfstring{Type $\A_n$ ($n \geqslant 1$)}{Type A}}\label{sec:proof-prop:A-type-A}
\begin{pa}
Let $d$ be a divisor of $n+1$ then we denote by $\bW(\times^d{\A_{(n+1-d)/d}})$ the parabolic subgroup of $\bW$ generated by the reflections $\mathbb{S}_0\setminus\{s_{j(n+1)/d} \mid 0 \leqslant j \leqslant d-1\}$.
\end{pa}

\begin{prop}[Lusztig, {\cite[3.2]{lusztig:2009:unipotent-classes-and-special-Weyl}}]\label{prop:lusztig-typeA}
Let $\mathcal{O} \in \Clu(\bG)^F$ be a unipotent class with class representative $u \in \mathcal{O}^F$ and denote by $d$ the order of $|A_{\bG}(u)|$. There exists a special character $[\Lambda] \in \Irr(\bW(\A_{(n+1-d)/d}))$ such that
\begin{itemize}
	\item $f_{\boxtimes^d[\Lambda]} = |A_{\btG}(u)|$
	\item $|\Stab_{\mathcal{A}}(\boxtimes^d [\Lambda])| = |Z_{\bG}(u)|$
	\item $j_{\bW(\times^d\A_{(n+1-d)/d})}^{\bW}(\boxtimes^d[\Lambda]) = \rho(\mathcal{O})$
\end{itemize}
\end{prop}

\begin{pa}
By the results in \cref{sec:quasi-isolated} there exists a semisimple element $\tilde{s}\in\btT_0^{\star}$ lying in an $F^{\star}$-stable conjugacy class such that $\bW(\tilde{s}) = \bW(\times^d\A_{(n+1-d)/d})$. We have $(\tilde{s},\bW^{\star}(\mathcal{\tilde{F}})) \in \mathcal{T}_{\btG}$ because $\boxtimes^d[\Lambda]$ is invariant under all graph automorphisms. We now prove \cref{P2}. By \cref{prop:lusztig-typeA} and the construction of $\psi$ we have $\psi$ is invariant under all possible graph automorphisms hence $\Stab_{A_{G^{\star}}(s)}(\psi) = A_{G^{\star}}(s)$. From \cref{sec:comp-orders-type-A} we see that $|A_{\bG^{\star}}(s)| = |Z_{\bG}(u)|$ hence \cref{P2} holds by \cref{lem:comp-group-fixed-point}. Finally \cref{P4} holds by combining \cref{prop:lusztig-j-ind} together with the fact that $b(\rho) = d(\rho)$ for all characters $\rho \in \Irr(\bW)$ and \cref{P5} holds trivially as $\mathcal{F}$ contains only one character.
\end{pa}

\section{\texorpdfstring{Type $\B_n$ ($n \geqslant 2$)}{Type B}}
\begin{pa}
Let $\mathbb{V}_n^1 = \{(\kappa,\mu,\nu) \mid \kappa+\mu+\nu = n$ and $\mu=0\} \subset \mathbb{N}^3$ and $\mathbb{V}_n^2 = \{(\kappa,\mu,\nu) \mid \kappa+\mu+\nu = n$ and $\kappa = \nu\} \subset \mathbb{N}^3$. If $(\kappa,\mu,\nu) \in \mathbb{V}_n^1$ then we denote by $\bW(\C_{\kappa}\A_{\mu}\C_{\nu}) = \bW(\C_{\kappa}\C_{\nu})$ the parahoric subgroup generated by the reflections $\mathbb{S}_0\setminus\{s_{\kappa}\}$. If $(\kappa,\mu,\nu) \in \mathbb{V}_n^2$ then we denote by $\bW(\C_{\kappa}\A_{\mu}\C_{\nu})$ the parahoric subgroup of $\bW$ generated by the reflections $\mathbb{S}_0\setminus\{s_{\mu},s_{n-\mu}\}$.
\end{pa}

\begin{prop}[H\'{e}zard, {\cite[\S 4.2.2]{hezard:2004:thesis}}]\label{prop:hezard-typeB}
Assume $\bG$ is adjoint and let $\mathcal{O} \in \Clu(\bG)^F$ be a unipotent class with class representative $u \in \mathcal{O}^F$. There exists a triple $(\kappa,\mu,\nu) \in \mathbb{V}_n^1$ and a special irreducible character $[\Lambda_1] \boxtimes [\Lambda_2] \in \Irr(\bW(\C_{\kappa}\C_{\nu}))$ such that:
\begin{itemize}
	\item $f_{[\Lambda_1]\boxtimes [\Lambda_2]} = |A_{\btG}(u)|$,
	\item $j_{\bW(\C_{\kappa}\C_{\nu})}^{\bW}([\Lambda_1] \boxtimes [\Lambda_2]) = \rho(\mathcal{O})$.
\end{itemize}
\end{prop}

\begin{prop}[Lusztig, {\cite[4.9 - 4.12]{lusztig:2009:unipotent-classes-and-special-Weyl}}]\label{prop:lusztig-typeB}
Assume $\bG$ is simply connected and let $\mathcal{O} \in \Clu(\bG)^F$ be a unipotent class with class representative $u \in \mathcal{O}^F$. There exists a triple $(\kappa,\mu,\nu) \in \mathbb{V}_n^{|Z_{\bG}(u)|}$ and a special irreducible character $[\Lambda_1] \boxtimes [\Lambda_2]\boxtimes [\Lambda_3] \in \Irr(\bW(\C_{\kappa}\A_{\nu}\C_{\mu}))$ such that:
\begin{itemize}
	\item $f_{[\Lambda_1]\boxtimes[\Lambda_2] \boxtimes [\Lambda_3]} = |A_{\btG}(u)|$,
	\item $|\Stab_{\mathcal{A}}([\Lambda_1]\boxtimes [\Lambda_2] \boxtimes [\Lambda_3])| = |Z_{\bG}(u)|$,
	\item $j_{\bW(\C_{\kappa}\A_{\nu}\C_{\mu})}^{\bW}([\Lambda_1] \boxtimes [\Lambda_2] \boxtimes [\Lambda_3]) = \rho(\mathcal{O})$,
	\item if $|Z_{\bG}(u)| = 2$ then $[\Lambda_1] = [\Lambda_3]$.
\end{itemize}
\end{prop}

\begin{pa}
By the results in \cref{sec:quasi-isolated} there exists a semisimple element $\tilde{s}$ lying in an $F^{\star}$-stable conjugacy class such that $\bW(\tilde{s}) = \bW(\C_{\kappa}\A_{\mu}\C_{\mu})$ with $(\kappa,\mu,\nu) \in \mathbb{V}_n^1\cup\mathbb{V}_n^2$. Furthermore if $(\kappa,\mu,\nu) \in \mathbb{V}_n^1$ then we may assume that $F^{\star}(\tilde{s}) \in \{\tilde{s},\tilde{s}^{w_0}\}$, (see \cref{cor:quasi-isolated-action}). In each case we see that either $F_{\tilde{s}}^{\star}$ acts trivially on $\bW^{\star}(\tilde{s})$ or the special character is invariant under all possible graph automorphisms, hence $(\tilde{s},\bW^{\star}(\mathcal{\tilde{F}})) \in \mathcal{T}_{\btG}$. This proves \cref{prop:A} when $\bG$ is adjoint. When $\bG$ is simply connected \cref{P2} follows from \cref{prop:lusztig-typeB}, the construction of $\psi$ and \cref{lem:comp-group-fixed-point}.

Let us now consider \cref{P5}. If $|Z_{\bG}(u)| = 1$ then this is clear unless possibly $n \equiv 0 \pmod{2}$ and $\mu = \nu = n/2$, in which case there is a non-trivial action of $A_{G^{\star}}(s)$ exchanging the two components of type $\C_{n/2}$. However \cref{prop:lusztig-typeB} implies that the families containing the characters $[\Lambda_1]$ and $[\Lambda_3]$ are distinct because if they weren't we would have $2 = |\Stab_{\mathcal{A}}([\Lambda_1]\boxtimes[\Lambda_2]\boxtimes[\Lambda_3])| \neq |Z_{\bG}(u)|$. Hence all the characters in $\mathcal{F}$ have a trivial stabiliser so \cref{P5} holds. Note that \cref{P5} doesn't necessarily hold if $|Z_{\bG}(u)|=2$, (take two distinct unipotent characters on the components of type $\C$), but this is one of the exceptions mentioned in \cref{prop:A}.
\end{pa}

\section{\texorpdfstring{Type $\C_n$ ($n \geqslant 3$)}{Type C}}
\begin{pa}
If $(\mu,\nu) \in \mathbb{N}^2$ are such that $\mu+\nu=n$ then we denote by $\bW(\D_{\mu}\B_{\nu})$ the parahoric subgroup of $\bW$ generated by the reflections $\mathbb{S}_0\setminus\{s_{\mu}\}$.
\end{pa}

\begin{prop}[H\'{e}zard, {\cite[\S 4.2.3]{hezard:2004:thesis}}]\label{prop:hez-typeC}
Let $\mathcal{O} \in \Clu(\bG)^F$ be a unipotent class with representative $u \in \mathcal{O}^F$. There exist $\mu,\nu \in \mathbb{N}$, such that $\mu + \nu = n$, and a special irreducible character $[\Lambda_1] \boxtimes [\Lambda_2] \in \Irr(\bW(\D_{\mu}\B_{\nu}))$ such that:
\begin{itemize}
	\item $f_{[\Lambda_1]\boxtimes[\Lambda_2]} = |A_{\btG}(u)|$,
	\item $j_{\bW(\D_{\mu}\B_{\nu})}^{\bW}([\Lambda_1] \boxtimes [\Lambda_2]) = \rho(\mathcal{O})$.
\end{itemize}
Furthermore if $|Z_{\bG_{\simc}}(u_{\simc})|=1$ then we have $\mu=0$ and if $|Z_{\bG_{\simc}}(u_{\simc})| = 2$ then the symbol $[\Lambda_1]$ is non-degenerate.
\end{prop}

\begin{pa}
By the results in \cref{sec:quasi-isolated} there exists a semisimple element $\tilde{s}\in\btT_0^{\star}$ such that $F^{\star}(\tilde{s}) \in \{\tilde{s},\tilde{s}^{w_0}\}$, (see \cref{cor:quasi-isolated-action}), and $\bW(\tilde{s}) = \bW(\D_{\mu}\B_{\nu})$. We have $(\tilde{s},\bW^{\star}(\mathcal{\tilde{F}})) \in \mathcal{T}_{\btG}$ as $F_{\tilde{s}}^{\star}$ acts trivially on $\bW^{\star}(\tilde{s})$. If $\bG$ is adjoint then this is all we need to show. Assume now that $\bG$ is simply connected then \cref{P2} follows from \cref{prop:hez-typeC}, the construction of $\psi$ and \cref{lem:comp-group-fixed-point}. Furthermore \cref{P5} clearly holds as \cref{prop:hez-typeC} says that all unipotent characters of the type $\D$ component in $\mathcal{F}$ are non-degenerate.
\end{pa}

\section{\texorpdfstring{Type $\D_n$ ($n \geqslant 4$)}{Type D}}\label{chap:prop-A-type-D}
\begin{rem}
The subgroup of $\bW$ which we denote by $\mathcal{A}$ is denoted by $\Omega$ in \cite{lusztig:2009:unipotent-classes-and-special-Weyl}. We now describe how Lusztig's chosen generators for $\Omega$ can be identified with our chosen generators for $\mathcal{A}$. Firstly, let us identify $\bW$ with a subgroup of the symmetric group on $\{1,\dots,n,n',\dots,1'\}$ as in \cite[5.1]{lusztig:2009:unipotent-classes-and-special-Weyl}. Lusztig fixes two generators $\omega_1$ and $\omega_2$ of $\Omega$ such that: if $n$ is even $\omega_1$ maps $i \mapsto (n+1-i)'$ and $i' \mapsto n+1-i$ for all $1 \leqslant i \leqslant n$, if $n$ is odd then $\omega_1$ maps $i \mapsto (n+1-i)'$ and $i' \mapsto n+1-i$ for all $1 \leqslant i \leqslant n-1$ and maps $n \mapsto 1$ and $n' \mapsto 1'$, finally $\omega_2$ maps $i \mapsto i$ for all $2 \leqslant i \leqslant n-1$ and interchanges $1$ with $1'$ and $n$ with $n'$. From this description it is not difficult to check that $\omega_2$ is identified with $\sigma_1$. Furthermore if $n \equiv 0 \pmod{2}$ then $\omega_1$ is identified with $\sigma_{n-1}$ and if $n \equiv 1\pmod{2}$ then $\omega_1$ is identified with $\sigma_n$. This now removes any difficulty in verifying our statements of Lusztig's results.
\end{rem}

\begin{pa}
Let $\mu,\nu \in \mathbb{N}$ then if $\mu+\nu = n$ and $\mu,\nu\geqslant 2$ we denote by $\bW(\D_{\mu}\D_{\nu})$ the parahoric subgroup generated by the reflections $\mathbb{S}_0\setminus\{s_{\mu}\}$. If $\mu=1$ we denote by $\bW(\D_{\mu}\D_{\nu})$ the parahoric subgroup generated by the reflections $\mathbb{S}_0\setminus\{s_0,s_1\}$, similarly if $\nu=1$ we take the set $\mathbb{S}_0\setminus\{s_{n-1},s_n\}$. If $2\mu+\nu = n$ and $\mu\geqslant 2$ we denote by $\bW(\D_{\mu}\A_{\nu}\D_{\mu})$ the parahoric subgroup generated by the reflections $\mathbb{S}_0\setminus\{s_{\mu},s_{n-\mu}\}$. If $\mu = 1$ then we denote by $\bW(\D_{\mu}\A_{\nu}\D_{\mu})$ the parabolic subgroup generated by the reflections $\mathbb{S}_0\setminus\{s_0,s_1,s_{n-1},s_n\}$. Finally we denote by $\bW(\A_{n-1}^+)$ the parabolic subgroup generated by the reflections $\mathbb{S}_0\setminus\{s_0,s_n\}$ and by $\bW(\A_{n-1}^-)$ the parabolic subgroup generated by the reflections $\mathbb{S}_0\setminus\{s_0,s_{n-1}\}$, (this is in keeping with the notation introduced in \cref{sec:clarification-typeD}).
\end{pa}

\begin{pa}
For the case of type $\D$ it will be useful for us to introduce some more notation for distinguishing unipotent classes. Assume $\mathcal{O} \in \Clu(\bG)$ with class representative $u \in \mathcal{O}$ such that $\mathcal{O}$ is parameterised by a partition $\lambda \vdash 2n$. Let us write $\lambda = (1^{r_1},2^{r_2},3^{r_3},\dots)$ with $r_i = 0$ when $i$ does not occur in $\lambda$ then we define the following two integers:
\begin{align*}
\kappa_{\D}(u) &= \max\{r_i \mid i\text{ is odd}\},\\
\delta_{\D}(u) &= \begin{cases}
1 & \text{if there exists an odd number }i\in\mathbb{N}\text{ such that }r_i \equiv 1\pmod{2},\\
0 & \text{otherwise}.
\end{cases}
\end{align*}
Note that in the notation of \cite{lusztig:2009:unipotent-classes-and-special-Weyl} we have for an element $y_* \in Y_m^n$ corresponding to $\mathcal{O}$ that $\delta_{\D}(u) = \delta_{y_*}$ and $\kappa_{\D}(u) = \max\{|\mathcal{I}| \mid \mathcal{I} \in \mathfrak{J}(y_*)\}$.
\end{pa}

\subsection{The Adjoint and Special Orthogonal Cases}
\begin{pa}\label{pa:triality-D4}
Until otherwise specified we assume that $\bG$ is either adjoint or a special orthogonal group. Let us first deal with the case where $\bG$ is of type $\D_4$ and $\tau$ induces the graph automorphism of order 3. The following table lists the classes $\mathcal{O} \in \Clu(\bG)^F$ by their partition and prescribes for each class the data described in \cref{sec:gen-strategy}. Note that the information for $|A_{\bG}(u)|$ is only for the case when $\bG$ is a special orthogonal group. When $\bG$ is adjoint we have $|A_{\bG}(u)| = |A_{\btG}(u)|$.
\begin{longtable}{>{$}c<{$}>{$}c<{$}>{$}c<{$}>{$}c<{$}>{$}c<{$}}
\toprule
\mathcal{O} & |A_{\btG}(u)| & |A_{\bG}(u)| & \bW^{\star}(\tilde{s}) & \chi \tabularnewline
\midrule
\endhead
\bottomrule
\endfoot
\bottomrule
\endlastfoot
(1,7) & 1 & 2 & \D_3 & [0;3] \tabularnewline
(3,5) & 1 & 2 & \D_3 & [1;2] \tabularnewline
(1^2,3^2) & 2 & 2 & \D_4 & [02;13] \tabularnewline
(1,2^2,3) & 1 & 2 & \D_2\D_2 & [01;12]\boxtimes [01;12] \tabularnewline
(1^4,2^2) & 1 & 1 & \D_4 & [013;123] \tabularnewline
(1^8) & 1 & 1 & \D_4 & [0123;1234]
\end{longtable}

\noindent By the results in \cref{sec:quasi-isolated} there exists a semisimple element $\tilde{s}\in\btT_0^{\star}$ lying in an $F^{\star}$-stable conjugacy class such that $\bW(\tilde{s}) = \bW(\D_4)$, $\bW(\D_3)$ or $\bW(\D_2\D_2)$. Clearly $(\tilde{s},\bW^{\star}(\mathcal{\tilde{F}})) \in \mathcal{T}_{\btG}$ as the character $\chi$ listed above is invariant under all graph automorphisms. If $\bG$ is a special orthogonal group then the non-trivial element of $A_{\bG^{\star}}(s)$ induces the graph automorphism $\sigma_1 \in \mathcal{A}$ therefore as non of the characters are degenerate we see that \cref{P2} holds by the construction of $\psi$ and \cref{lem:comp-group-fixed-point}. \cref{P5} is also clear for the same reason which deals with the triality case.
\end{pa}

\begin{prop}[H\'{e}zard, {\cite[\S 4.2.4]{hezard:2004:thesis}}]\label{prop:hez-typeD}
Assume $\tau$ has order at most 2 and let $\mathcal{O} \in \Clu(\bG)^F$ be a unipotent class with class representative $u \in \mathcal{O}^F$. There exist $\mu$, $\nu \in \mathbb{N}$, such that $\mu+\nu=n$, and a special irreducible character $[\Lambda_1]\boxtimes[\Lambda_2] \in \Irr(\bW(\D_{\mu}\D_{\nu}))$ such that:
\begin{itemize}
	\item $f_{[\Lambda_1]\boxtimes [\Lambda_2]} = |A_{\btG}(u)|$,
	\item $j_{\bW(\D_{\mu}\D_{\nu})}^{\bW}([\Lambda_1]\boxtimes[\Lambda_2]) = \rho(\mathcal{O})$.
\end{itemize}
Furthermore the symbols $[\Lambda_1]$ and $[\Lambda_2]$ are non-degenerate unless $\mathcal{O}$ is a degenerate unipotent class, in which case we have $\mu = 0$ and $[\Lambda_1] = \rho(\mathcal{O})$ is degenerate.
\end{prop}

\begin{pa}
By the results in \cref{sec:quasi-isolated} there exists a semisimple element $\tilde{s}\in\btT_0^{\star}$ such that $F^{\star}(\tilde{s}) \in \{\tilde{s},\tilde{s}^{w_0}\}$, (see \cref{cor:quasi-isolated-action}), and $\bW(\tilde{s}) = \bW(\D_{\mu}\D_{\nu})$. We claim $(\tilde{s},\bW^{\star}(\mathcal{\tilde{F}})) \in \mathcal{T}_{\btG}$, which follows from the following two facts. Firstly $F_{\tilde{s}}^{\star}$ cannot act by exchanging the two components of type $\D$. Secondly a unipotent class $\mathcal{O} \in \Clu(\bG)$ is $F$-stable unless $\mathcal{O}$ is degenerate and $\tau$ is of order 2 in which case it is not $F$-stable, (see for instance \cite[Lemma 2.40]{taylor:2012:thesis}). This together with \cref{prop:hez-typeD} confirms our claim, hence this proves \cref{prop:A} in the adjoint case.

Let us now consider \cref{P2} when $\bG$ is a special orthogonal group. By \cite[\S10.6]{lusztig:1984:intersection-cohomology-complexes} and \cite[\S14.3]{lusztig:1984:intersection-cohomology-complexes} we have $|Z_{\bG}(u)| = 2$ unless $\mathcal{O}$ is degenerate, in which case $|Z_{\bG}(u)| = 1$. If $\mathcal{O}$ is degenerate then \cref{P2} is clear, so let us assume $\mathcal{O}$ is non-degenerate. The non-trivial element of $A_{\bG^{\star}}(s)$ induces the graph automorphism $\sigma_1 \in \mathcal{A}$, therefore $\Stab_{A_{G^{\star}}(s)}(\psi) = A_{G^{\star}}(s)$ and \cref{P2} follows from \cref{lem:comp-group-fixed-point}. Finally, as $[\Lambda_1]$, $[\Lambda_2]$ are not degenerate this is true of all characters in the families $\mathcal{F}_1$ and $\mathcal{F}_2$. In particular the stabiliser of any character in the family $\mathcal{F}$ has the same order as $\Stab_{A_{G^{\star}}(s)}(\psi)$ so \cref{P5} holds which proves \cref{prop:A} when $\bG$ is a special orthogonal group.
\end{pa}

\subsection{The Simply Connected and Half-Spin Cases}
\begin{pa}
We assume from now until the end of this section that $\bG$ is either simply connected or a half-spin group.
\end{pa}

\begin{prop}[Lusztig, {\cite[6.14]{lusztig:2009:unipotent-classes-and-special-Weyl}}]\label{prop:lusztig-typeD-1}
Let $\mathcal{O} \in \Clu(\bG)^F$ be a unipotent class with class representative $u \in \mathcal{O}^F$ such that $\kappa_{\D}(u) \geqslant 2$ and $\delta_{\D}(u) = 1$. There exists $\mu,\nu \in \mathbb{N}$, such that $\mu + \nu = n$, and a special character $[\Lambda_1] \boxtimes [\Lambda_2] \in \Irr(\bW(\D_{\mu}\D_{\nu}))$ such that:
\begin{itemize}
	\item $f_{[\Lambda_1]\boxtimes[\Lambda_2]} = |A_{\btG}(u)|$,
	\item $\Stab_{\mathcal{A}}([\Lambda_1]\boxtimes [\Lambda_2]) = \langle \sigma_1 \rangle$
	\item $j_{\bW(\D_{\mu}\D_{\nu})}^{\bW}([\Lambda_1] \boxtimes [\Lambda_2]) = \rho(\mathcal{O})$.
\end{itemize}
\end{prop}

\begin{pa}
Let us note first of all that if $\tau$ is of order 3 then there are no $F$-stable classes satisfying the conditions of \cref{prop:lusztig-typeD-1}, (see \cref{pa:triality-D4}). Hence we may assume that $\tau$ has order at most two. By the results in \cref{sec:quasi-isolated} there exists a semisimple element $\tilde{s} \in \btT_0^{\star}$ such that $F^{\star}(\tilde{s}) \in \{\tilde{s},\tilde{s}^{w_0}\}$, (see \cref{cor:quasi-isolated-action}), and $\bW(\tilde{s}) =\bW(\D_{\mu}\D_{\nu})$. We claim $(\tilde{s},\bW^{\star}(\mathcal{\tilde{F}})) \in \mathcal{T}_{\btG}$, which again follows from the following two facts. Firstly $F_{\tilde{s}}^{\star}$ cannot exchange the two components of type $\D$. Secondly $[\Lambda_1]$ and $[\Lambda_2]$ are not degenerate because, by \cref{prop:lusztig-typeD-1}, they are invariant under the graph automorphism induced by $\sigma_1$.

We now concern ourselves with \cref{P2}, which we must consider for each isomorphism type. Firstly by \cite[\S10.6]{lusztig:1984:intersection-cohomology-complexes} and \cite[\S14.3]{lusztig:1984:intersection-cohomology-complexes} we have
\begin{equation*}
|Z_{\bG}(u)| = \begin{cases}
1 &\text{if }\bG\text{ is a half-spin group},\\
2 &\text{if }\bG\text{ is simply connected}.
\end{cases}
\end{equation*}
If $\bG$ is simply connected then $Z(\bG)$ is not necessarily cyclic, which means we cannot use \cref{lem:comp-group-fixed-point} so we must argue directly. As $\Stab_{\mathcal{A}}([\Lambda_1]\boxtimes [\Lambda_2]) \neq \mathcal{A}$ we must have $[\Lambda_1]\neq[\Lambda_2]$ so $|\Stab_{A_{G^{\star}}(s)}(\psi)| \leqslant 2$ hence we need only show that $A_{\bG^{\star}}(s)^{F_{\tilde{s}}^{\star}} = A_{\bG^{\star}}(s)$ and $Z_{\bG}(u)^F = Z_{\bG}(u)$. If $F$ acts trivially on $Z(\bG)$ then $F^{\star}$ acts trivially on $\Ker(\delta_{\simc}^{\star})$ so everything is fixed. Assume $F$ acts non-trivially on the centre then it acts by exchanging $\hat{z}_{n-1}$, $\hat{z}_n$ hence we will have $Z(\bG)^F = \langle \hat{z}_1 \rangle = \Ker(\delta_{\simc}^{\star})^{F^{\star}}$, (where here we have identified $Z(\bG)$ with $Z(\bG_{\ad}^{\star})$). It is clear that $A_{\bG^{\star}}(s)$ is isomorphic to the subgroup $\langle \hat{z}_1 \rangle$ so $A_{\bG^{\star}}(s)^{F_{\tilde{s}}^{\star}}=A_{\bG^{\star}}(s)$. From the discussion in \cite[\S2.2.4 - pg.\ 63]{taylor:2012:thesis} we have $Z_{\bG}(u)$ is given by $C_{\bG}(u)^{\circ}$ and $\hat{z}_{n-1}C_{\bG}(u)^{\circ} = \hat{z}_nC_{\bG}(u)^{\circ}$ so we also have $Z_{\bG}(u)^F = Z_{\bG}(u)$.

Assume now that $\bG$ is a half-spin group. If $\mu\neq\nu$ then $|A_{\bG^{\star}}(s)|=1$ and the result is clear. If $\mu=\nu$ then $|A_{\bG^{\star}}(s)|=2$ however the graph automorphism induced by the non-trivial element of $A_{\bG^{\star}}(s)$ exchanges the two components of type $\D_{n/2}$. As was mentioned above $[\Lambda_1]\neq[\Lambda_2]$ so $\Stab_{A_{G^{\star}}(s)}(\psi)$ is trivial. Finally, as $[\Lambda_1]$, $[\Lambda_2]$ are not degenerate this is true of all characters in the families $\mathcal{F}_1$ and $\mathcal{F}_2$. In particular the stabiliser of any character in the family $\mathcal{F}$ has the same order as $\Stab_{A_{G^{\star}}(s)}(\psi)$ so \cref{P5} holds.
\end{pa}

\begin{prop}[Lusztig, {\cite[6.11]{lusztig:2009:unipotent-classes-and-special-Weyl}}]\label{prop:lusztig-typeD-2}
Let $\mathcal{O} \in \Clu(\bG)^F$ be a unipotent class with class representative $u \in \mathcal{O}^F$ such that $\kappa_{\D}(u) = 1$ and $\delta_{\D}(u) = 1$. There exist $\mu, \nu \in \mathbb{N}$, such that $2\mu + \nu = n$, and a special character $[\Lambda_1]\boxtimes[\Lambda_2]\boxtimes[\Lambda_1] \in \Irr(\bW(\D_{\mu}\A_{\nu}\D_{\mu}))$ such that:
\begin{itemize}
	\item $f_{[\Lambda_1]\boxtimes[\Lambda_2] \boxtimes [\Lambda_1]} = |A_{\btG}(u)|$,
	\item $\Stab_{\mathcal{A}}([\Lambda_1]\boxtimes [\Lambda_2] \boxtimes [\Lambda_1]) = \mathcal{A}$
	\item $j_{\bW(\D_{\mu}\A_{\nu}\D_{\mu})}^{\bW}([\Lambda_1] \boxtimes [\Lambda_2] \boxtimes [\Lambda_1]) = \rho(\mathcal{O})$.
\end{itemize}
\end{prop}

\begin{pa}
By the results in \cref{sec:quasi-isolated} there exists a semisimple element $\tilde{s} \in \btT_0^{\star}$ contained in an $F^{\star}$-stable conjugacy class such that $\bW(\tilde{s})=\bW(\D_{\mu}\A_{\nu}\D_{\mu})$. We have $(\tilde{s},\bW^{\star}(\tilde{\mathcal{F}})) \in \mathcal{T}_{\btG}$ because $[\Lambda_1]\boxtimes [\Lambda_2] \boxtimes [\Lambda_1]$ is invariant under all graph automorphisms. We now concern ourselves with \cref{P2}, which we must consider for each isomorphism type. Firstly by \cite[\S10.6]{lusztig:1984:intersection-cohomology-complexes} and \cite[\S14.3]{lusztig:1984:intersection-cohomology-complexes} we have
\begin{equation*}
|Z_{\bG}(u)| = \begin{cases}
2 &\text{if }\bG\text{ is a half-spin group},\\
4 &\text{if }\bG\text{ is simply connected}.
\end{cases}
\end{equation*}
In both cases we have $Z_{\bG}(u) \cong Z(\bG)$ and $A_{\bG^{\star}}(s) \cong \Ker(\delta_{\simc}^{\star})$ so \cref{P2} holds by \cref{cor:centre-duality} and the fact that $\psi$ is invariant under all graph automorphisms. Let us now consider \cref{P5}. Every character $\chi \in \mathcal{F}$ is parameterised by symbols $[\Lambda_1] \boxtimes [\Lambda_2] \boxtimes [\Lambda_3]$, where neither $[\Lambda_1]$ nor $[\Lambda_3]$ are degenerate. However we have $|\Stab_{A_{G^{\star}}(s)}(\chi)| = |Z_{\bG}(u)^F|$ if and only if $[\Lambda_1] = [\Lambda_3]$. Hence it is clear that \cref{P5} does not always hold, however this is one of the exceptions mentioned in \cref{prop:A}.
\end{pa}

\begin{prop}[Lusztig, {\cite[6.13]{lusztig:2009:unipotent-classes-and-special-Weyl}}]\label{prop:lusztig-typeD-3}
Let $\mathcal{O} \in \Clu(\bG)^F$ be a unipotent class with class representative $u \in \mathcal{O}^F$ such that $\kappa_{\D}(u) = 0$, i.e.\ $\mathcal{O}$ is a degenerate unipotent class. There exists a parabolic subgroup $\bW'$, which is either $\bW(\A_{n-1}^+)$ or $\bW(\A_{n-1}^-)$, and a special irreducible character $[\Lambda] \in \Irr(\bW')$ such that:
\begin{itemize}
	\item $f_{[\Lambda]} = |A_{\btG}(u)|$,
	\item $\Stab_{\mathcal{A}}([\Lambda]) = \begin{cases} \langle \sigma_n \rangle &\text{if }[\Lambda] \in \Irr(\bW(\A_{n-1}^+)),\\ \langle \sigma_{n-1} \rangle &\text{if }[\Lambda] \in \Irr(\bW(\A_{n-1}^-)),\end{cases}$
	\item $j_{\bW'}^{\bW}([\Lambda]) = \rho(\mathcal{O})$.
\end{itemize}
\end{prop}

\begin{rem}\label{rem:degen-unip-class}
In \cite[6.3]{lusztig:2009:unipotent-classes-and-special-Weyl} Lusztig describes two ways to construct a parabolic subgroup of $\bW$ which is of type $\A_{n-1}$ and he denotes these by $W_0'\times S_n^{(\lambda)} \times W_0'$ with $\lambda = 0$ or 3. It is easily checked that the parabolic subgroup with $\lambda=0$ corresponds to $\bW(\A_{n-1}^+)$ and the parabolic subgroup with $\lambda=3$ corresponds to $\bW(\A_{n-1}^-)$. Unless $\bG$ is a half-spin group the choice of parabolic subgroup is immaterial to our result. However when $\bG$ is a half-spin group this choice is important and we will explain below that it can be made concrete by the transitivity of $j$-induction.
\end{rem}

\begin{rem}
As mentioned above we have a degenerate unipotent class is $F$-stable if and only if $\tau$ is trivial. In particular we may assume until the end of this section that $F$ is such a Frobenius endomorphism.
\end{rem}

\begin{pa}
By the results in \cref{sec:quasi-isolated} there exists a semisimple element $\tilde{s} \in \btT_0^{\star}$ which lies in an $F^{\star}$-stable conjugacy class such that $\bW(\tilde{s}) = \bW(\A_{n-1}^{\pm})$. We have $(\tilde{s},\bW^{\star}(\tilde{\mathcal{F}})) \in \mathcal{T}_{\btG}$ because $\bW^{\star}(\tilde{\mathcal{F}})$ contains only the character $[\Lambda]$ and any irreducible character of a Weyl group of type $\A$ is invariant under all graph automorphisms.
\end{pa}

\begin{pa}
We now concern ourselves with \cref{P2}, which we must consider for each isomorphism type. Assume $\bG$ is simply connected then by \cite[\S14.3]{lusztig:1984:intersection-cohomology-complexes} we have $|Z_{\bG}(u)|=2$ and by the results in \cref{sec:quasi-isolated} we have $|A_{\bG^{\star}}(s)|=2$, so we at least have $|A_{\bG^{\star}}(s)| = |Z_{\bG}(u)|$. The Frobenius endomorphism $F_q$ cannot exchange the elements $\hat{z}_{n-1}$ and $\hat{z}_n$ because these are both elements of order 2 and $q$ is odd, it then follows that $Z(\bG)^F = Z(\bG)$ and $\Ker(\delta_{\simc}^{\star})^{F^{\star}} = \Ker(\delta_{\simc}^{\star})$. Hence $|A_{\bG^{\star}}(s)^{F_s^{\star}}| = |Z_{\bG}(u)^F|$ and by \cref{prop:lusztig-typeD-3} we have $\Stab_{A_{G^{\star}}(s)}(\psi) = A_{G^{\star}}(s)$ so \cref{P2} holds.
\end{pa}

\begin{pa}
Let us now deal with the case where $\bG$ is a half-spin group. We may assume $\mathcal{O}$ is of the form $\mathcal{O}_{\lambda}^{\pm}$ and $\eta \vdash n/2$ is the partition constructed from $\lambda$ in \cref{pa:degen-springer-corr}. We must first clarify the choice over the parabolic subgroup in \cref{prop:lusztig-typeD-3}. The subgroup $\bW(\A_{\eta^*}^{\pm})$ is a parabolic subgroup of $\bW(\A_{n-1}^{\pm})$ so by the transitivity of $j$-induction we have
\begin{equation*}
j_{\bW(\A_{\eta^*}^{\pm})}^{\bW}(\sgn) = j_{\bW(\A_{n-1}^{\pm})}^{\bW}(j_{\bW(\A_{\eta^*}^{\pm})}^{\bW(\A_{n-1}^{\pm})}(\sgn)).
\end{equation*}
Let $\bW(\A_{\eta^*}')$ denote either $\bW(\A_{\eta^*}^+)$ or $\bW(\A_{\eta^*}^-)$. By this remark we must have the parabolic subgroup $\bW'$ from \cref{prop:lusztig-typeD-3} contains the parabolic subgroup $\bW(\A_{\eta^*}')$ for which $\rho(\mathcal{O}) = j_{\bW(\A_{\eta^*}')}^{\bW}(\sgn)$. From \cref{eq:degen-spring-correspondence} we see that this depends upon the congruence of $n$ modulo 4 so we consider this in two cases. Before we consider the two cases let us first introduce the following notation. We denote by $\bW^{\star}(\A_{n-1}^+)$ the parabolic subgroup of $\bW^{\star}$ generated by the reflections $\mathbb{T}\setminus\{t_n\}$ and by $\bW^{\star}(\A_{n-1}^-)$ the parabolic subgroup of $\bW^{\star}$ generated by the reflections $\mathbb{T}\setminus\{t_{n-1}\}$. Similarly we denote by $\bW^{\star}(\A_{\eta*}^{\pm})$ the appropriate parabolic subgroup of $\bW^{\star}(\A_{n-1}^{\pm})$. 

Assume $n \equiv 0 \pmod{4}$ then the character $\rho(\mathcal{O}_{\lambda}^{\pm})$ is the induced character $j_{\bW(\A_{n-1}^{\pm})}^{\bW}(\sgn)$. Under the duality described in \cref{sec:clarification-typeD} we see that $s_i^{\star} = t_i$ for all $1 \leqslant i \leqslant n$. Hence under this isomorphism the subgroup $\bW(\A_{n-1}^{\pm})$ is sent to the subgroup $\bW^{\star}(\A_{n-1}^{\pm})$. In particular the element $\tilde{s} \in \btT_0^{\star}$ is such that $\bW^{\star}(\tilde{s}) = \bW^{\star}(\A_{n-1}^{\pm})$. From \cref{tab:quasi-isolated-Dn} we see that we always have $|A_{\bG^{\star}}(s)| = |Z_{\bG}(u)|$, therefore \cref{P2} holds by \cref{lem:comp-group-fixed-point} as $\Stab_{A_{G^{\star}}(s)}(\psi) = A_{G^{\star}}(s)$.

Assume $n \equiv 2 \pmod{4}$ then the character $\rho(\mathcal{O}_{\lambda}^{\pm})$ is the induced character $j_{\bW(\A_{\eta^*}^{\mp})}^{\bW}(\sgn)$. Under the duality described in \cref{sec:clarification-typeD} we see that $s_i^{\star} = t_i$ for all $1 \leqslant i \leqslant n-2$. However $s_{n-1}^{\star} = t_n$ and $s_n^{\star} = t_{n-1}$, hence under the isomorphism $\bW \to \bW^{\star}$ the subgroup $\bW(\A_{n-1}^{\mp})$ is sent to the subgroup $\bW^{\star}(\A_{n-1}^{\pm})$. In particular the element $\tilde{s} \in \btT_0^{\star}$ is such that $\bW^{\star}(\tilde{s}) = \bW^{\star}(\A_{n-1}^{\pm})$. From \cref{tab:quasi-isolated-Dn} we see that we always have $|A_{\bG^{\star}}(s)| = |Z_{\bG}(u)|$, therefore \cref{P2} holds by \cref{lem:comp-group-fixed-point} as $\Stab_{A_{G^{\star}}(s)}(\psi) = A_{G^{\star}}(s)$.

Note finally that \cref{P5} holds trivially as $\mathcal{F}$ contains only one character.
\end{pa}

\section{Exceptional Types}
\begin{pa}
In what follows we give tables for the simple groups of exceptional type containing the information described in \cref{sec:gen-strategy}, (for types $\G_2$, $\F_4$ and $\E_8$ this information is taken from \cite[\S5]{hezard:2004:thesis}). By the results in \cref{sec:quasi-isolated} there exists a semisimple element $\tilde{s} \in \btT_0^{\star}$ which lies in an $F^{\star}$-stable conjugacy class such that $\bW^{\star}(\tilde{s})$ is one of the parahoric subgroups listed below. One easily verifies that the special characters we list are invariant under all graph automorphisms hence $(\tilde{s},\bW^{\star}(\tilde{\mathcal{F}})) \in \mathcal{T}_{\btG}$.

Let us now consider \cref{P4}. If $\tilde{s}$ is isolated then this follows from \cite[Proposition 2.3]{geck-hezard:2008:unipotent-support} as in the case of classical type groups. The only cases where this is not the case are when $\bG$ is of type $\E_6$ or $\E_7$. In these cases one may easily check that \cref{P4} holds by using Geck's {\sf PyCox} program, (see \cite{geck:2012:pycox-computing}). In particular assume {\tt W} is the Weyl group $\bW$ and {\tt H} is the reflection subgroup $\bW(\tilde{s})$ then running the command
\begin{verbatim}
    >>> X = inductiontable(H,W,invchar=(lambda G:dimBu(G)))
\end{verbatim}
in {\sf PyCox} produces the truncated induction table with respect to the $d$-function. Assume \verb+j+ is the index of $\chi$ in the list \verb+X[`charH']+ then one only has to check that there is a unique \verb+i+ such that \verb+X[`scalar'][i][j]+ is non-zero. With this we see that \cref{P4} holds hence this completes the proof of \cref{prop:A} if $\bG$ is adjoint.
\end{pa}

\begin{pa}
Let us now consider \cref{P2} and \cref{P5} in the cases where $\bG$ is simply connected of type $\E_6$ or $\E_7$. From the tables below we see that $|A_{\bG^{\star}}(s)| = |Z_{\bG}(u)|$. By \cref{lem:invariance-unipotent} it is clear that the character $\psi$ will be invariant under any graph automorphism induced by an element of $A_{\bG^{\star}}(s)$ so we will have $\Stab_{A_{G^{\star}}(s)}(\psi) = A_{G^{\star}}(s)$ hence \cref{P2} holds by \cref{lem:comp-group-fixed-point}.

Let us now consider the validity of \cref{P5}. If $|A_{\btG}(u)| = 1$ then $|\mathcal{F}| = 1$ so it is obvious that $\mathfrak{X}_{\mathcal{F}} = \emptyset$. Assume $\mathcal{O}$ is such that $|A_{\btG}(u)| = |A_{\bG}(u)|$ then $\mathfrak{X}_{\mathcal{F}} = \emptyset$ because $C_{\bG^{\star}}(s)$ is connected so clearly $|\Stab_{A_{G^{\star}}(s)}(\psi)| = 1$ for all $\psi \in \mathcal{F}$. We need now only consider when $\mathcal{O}$ is the class $\E_6(a_3)$ in $\E_6$ or $\mathcal{O}$ is the class $\D_4(a_1) + \A_1$, $\E_7(a_3)$, $\E_7(a_4)$ or $\E_7(a_5)$ in $\E_7$. However by \cref{lem:invariance-unipotent} we know all unipotent characters in the family are invariant under all graph automorphisms which shows $\mathfrak{X}_{\mathcal{F}} = \emptyset$.
\end{pa}

\subsection{\texorpdfstring{Type $\G_2$}{Type G2}}
\begin{longtable}{>{$}c<{$}>{$}c<{$}>{$}c<{$}>{$}c<{$}}
\toprule
\mathcal{O} & |A_{\bG}(u)| & \bW^{\star}(\tilde{s}) & \chi \tabularnewline
\midrule
\endhead
\bottomrule
\endfoot
\bottomrule
\endlastfoot
1 & 1 & \G_2 & \phi_{1,6} \tabularnewline
\A_1 & 1 & \A_2 & [123] \tabularnewline
\tilde{\A}_1 & 1 & \A_1\A_1 & [12]\boxtimes[12] \tabularnewline
\G_2(a_1) & 6 & \G_2 & \phi_{2,1} \tabularnewline
\G_2 & 1 & \G_2 & \phi_{1,0}
\end{longtable}

\subsection{\texorpdfstring{Type $\F_4$}{Type F4}}
\begin{longtable}{>{$}c<{$}>{$}c<{$}>{$}c<{$}>{$}c<{$}}
\toprule
\mathcal{O} & |A_{\bG}(u)| & \bW^{\star}(\tilde{s}) & \chi \tabularnewline
\midrule
\endhead
\bottomrule
\endfoot
\bottomrule
\endlastfoot
1 & 1 & \F_4 & 1_4 \tabularnewline
\A_1 & 1 & \B_4 & [01234;1234] \tabularnewline
\tilde{\A}_1 & 2 & \F_4 & 4_5 \tabularnewline
\A_1+\tilde{\A}_1 & 1 & \F_4 & 9_4 \tabularnewline
\A_2 & 2 & \B_4 & [0124;123] \tabularnewline
\tilde{\A}_2 & 1 & \F_4 & 8_2 \tabularnewline
\A_2+\tilde{\A}_1 & 1 & \A_3\A_1 & [1234]\boxtimes[12] \tabularnewline
\B_2 & 2 & \B_4 & [023;12] \tabularnewline
\tilde{\A}_2+\A_1 & 1 & \A_2\A_2 & [123]\boxtimes[123] \tabularnewline
\C_3(a_1) & 2 & \B_4 & [03;2] \tabularnewline
\F_4(a_3) & 24 & \F_4 & 12 \tabularnewline
\B_3 & 1 & \F_4 & 8_1 \tabularnewline
\C_3 & 1 & \F_4 & 8_3 \tabularnewline
\F_4(a_2) & 2 & \B_4 & [013;13] \tabularnewline
\F_4(a_1) & 2 & \F_4 & 4_2 \tabularnewline
\F_4 & 1 & \F_4 & 1_1 \tabularnewline
\end{longtable}

\subsection{\texorpdfstring{Type $\E_6$}{Type E6}}
\begin{longtable}{>{$}c<{$}>{$}c<{$}>{$}c<{$}>{$}c<{$}>{$}c<{$}}
\toprule
\mathcal{O} & |A_{\btG}(u)| & |A_{\bG}(u)| & \bW^{\star}(\tilde{s}) & \chi \tabularnewline
\midrule
\endhead
\bottomrule
\endfoot
\bottomrule
\endlastfoot
1 & 1 & 1 & \E_6 & 1_p' \tabularnewline
\A_1 & 1 & 1 & \E_6 & 6_p' \tabularnewline
2\A_1 & 1 & 1 & \E_6 & 20_p' \tabularnewline
3\A_1 & 1 & 1 & \A_5\A_1 & [123456]\boxtimes[12] \tabularnewline
\A_2 & 2 & 2 & \E_6 & 30_p' \tabularnewline
\A_2 + \A_1 & 1 & 1 & \E_6 & 64_p' \tabularnewline
2\A_2 & 1 & 3 & \D_4 & [0123;1234] \tabularnewline
\A_2 + 2\A_1 & 1 & 1 & \E_6 & 60_p' \tabularnewline
\A_3 & 1 & 1 & \E_6 & 81_p' \tabularnewline
2\A_2 + \A_1 & 1 & 3 & \A_2\A_2\A_2 & [123]\boxtimes[123]\boxtimes[123] \tabularnewline
\A_3+\A_1 & 1 & 1 & \A_5\A_1 & [1245]\boxtimes[12] \tabularnewline
\D_4(a_1) & 6 & 6 & \E_6 & 80_s \tabularnewline
\A_4 & 1 & 1 & \E_6 & 81_p \tabularnewline
\D_4 & 1 & 1 & \E_6 & 24_p \tabularnewline
\A_4+\A_1 & 1 & 1 & \E_6 & 60_p \tabularnewline
\A_5 & 1 & 3 & \A_1\A_1\A_1\A_1 & [12]\boxtimes[12]\boxtimes[12]\boxtimes[12] \tabularnewline
\D_5(a_1) & 1 & 1 & \E_6 & 64_p \tabularnewline
\E_6(a_3) & 2 & 6 & \D_4 & [02;13] \tabularnewline
\D_5 & 1 & 1 & \E_6 & 20_p \tabularnewline
\E_6(a_1) & 1 & 3 & \D_4 & [1;3] \tabularnewline
\E_6 & 1 & 3 & \A_2\A_2\A_2 & [3]\boxtimes[3]\boxtimes[3]
\end{longtable}

\subsection{\texorpdfstring{Type $\E_7$}{Type E7}}
\begin{longtable}{>{$}c<{$}>{$}c<{$}>{$}c<{$}>{$}c<{$}>{$}c<{$}}
\toprule
\mathcal{O} & |A_{\btG}(u)| & |A_{\bG}(u)| & \bW^{\star}(\tilde{s}) & \chi \tabularnewline
\midrule
\endhead
\bottomrule
\endfoot
\bottomrule
\endlastfoot
1 & 1 & 1 & \E_7 & 1_a' \tabularnewline
\A_1 & 1 & 1 & \E_7 & 7_a \tabularnewline
2\A_1 & 1 & 1 & \E_7 & 27_a' \tabularnewline
(3\A_1)'' & 1 & 2 & \E_6 & 1_p' \tabularnewline
(3\A_1)' & 1 & 1 & \A_1\D_6 & [12]\boxtimes [012345;123456] \tabularnewline
\A_2 & 2 & 2 & \E_7 & 56_a \tabularnewline
4\A_1 & 1 & 2 & \A_7 & [12345678] \tabularnewline
\A_2 + \A_1 & 2 & 2 & \E_7 & 120_a' \tabularnewline
\A_2 + 2\A_1 & 1 & 1 & \E_7 & 189_b \tabularnewline
\A_3 & 1 & 1 & \E_7 & 210_a' \tabularnewline
2\A_2 & 1 & 1 & \E_7 & 168_a' \tabularnewline
\A_2 + 3\A_1 & 1 & 2 & \A_7 & [1234568] \tabularnewline
(\A_3+\A_1)'' & 1 & 2 & \E_6 & 20_p' \tabularnewline
2\A_2+\A_1 & 1 & 1 & \A_2\A_5 & [123]\boxtimes[123456] \tabularnewline
(\A_3+\A_1)' & 1 & 1 & \A_1\D_6 & [12]\boxtimes[0134;1234] \tabularnewline
\D_4(a_1) & 6 & 6 & \E_7 & 315_a \tabularnewline
\A_3+2\A_1 & 1 & 2 & \A_7 & [123467] \tabularnewline
\D_4 & 1 & 1 & \A_2\A_5 & [3]\boxtimes[123456] \tabularnewline
\D_4(a_1)+\A_1 & 2 & 4 & \E_6 & 30_p' \tabularnewline
\A_3 + \A_2 & 2 & 2 & \A_1\D_6 & [12]\boxtimes[0124;1235] \tabularnewline
\A_4 & 2 & 2 & \A_1\D_6 & [2]\boxtimes[0124;1235] \tabularnewline
\A_3+\A_2+\A_1 & 1 & 2 & \A_3\A_3\A_1 & [1234]\boxtimes[1234]\boxtimes[12] \tabularnewline
(\A_5)'' & 1 & 2 & \D_4\A_1\A_1 & [0123;1234] \boxtimes [2] \boxtimes [2] \tabularnewline
\D_4+\A_1 & 1 & 2 & \A_7 & [2345] \tabularnewline
\A_4+\A_1 & 2 & 2 & \E_7 & 512_a' \tabularnewline
\D_5(a_1) & 2 & 2 & \E_7 & 420_a \tabularnewline
\A_4+\A_2 & 1 & 1 & \E_7 & 210_b \tabularnewline
(\A_5)' & 1 & 1 & \A_1\D_6 & [12]\boxtimes [12;23] \tabularnewline
\A_5+\A_1 & 1 & 2 & \A_2\A_2\A_2 & [123]\boxtimes[123]\boxtimes[123] \tabularnewline
\D_5(a_1)+\A_1 & 1 & 2 & \A_7 & [1346] \tabularnewline
\D_6(a_2) & 1 & 2 & \A_7 & [1256] \tabularnewline
\E_6(a_3) & 2 & 2 & \E_7 & 405_a \tabularnewline
\D_5 & 1 & 1 & \E_7 & 189_c' \tabularnewline
\E_7(a_5) & 6 & 12 & \E_6 & 80_s \tabularnewline
\A_6 & 1 & 1 & \E_7 & 105_b \tabularnewline
\D_5+\A_1 & 1 & 2 & \A_7 & [236] \tabularnewline
\D_6(a_1) & 1 & 2 & \A_7 & [1238] \tabularnewline
\E_7(a_4) & 2 & 4 & \D_4\A_1\A_1 & [02;13] \boxtimes [12] \boxtimes [12] \tabularnewline
\D_6 & 1 & 2 & \A_7 & [4567] \tabularnewline
\E_6(a_1) & 2 & 2 & \E_7 & 120_a \tabularnewline
\E_6 & 1 & 1 & \E_7 & 21_b' \tabularnewline
\E_7(a_3) & 2 & 4 & \E_6 & 30_p \tabularnewline
\E_7(a_2) & 1 & 2 & \E_6 & 20_p \tabularnewline
\E_7(a_1) & 1 & 2 & \E_6 & 6_p \tabularnewline
\E_7 & 1 & 2 & \A_7 & [8]
\end{longtable}

\subsection{\texorpdfstring{Type $\E_8$}{Type E8}}
\begin{longtable}{>{$}c<{$}>{$}c<{$}>{$}c<{$}>{$}c<{$}}
\toprule
\mathcal{O} & |A_{\bG}(u)| & \bW^{\star}(\tilde{s}) & \chi \tabularnewline
\midrule
\endhead
\bottomrule
\endfoot
\bottomrule
\endlastfoot
1 & 1 & \E_8 & 1_x' \tabularnewline
\A_1 & 1 & \E_8 & 8_z' \tabularnewline
2\A_1 & 1 & \E_8 & 35_x' \tabularnewline
3\A_1 & 1 & \E_7\A_1 & 1_a'\boxtimes [12] \tabularnewline
\A_2 & 2 & \E_8 & 112_z' \tabularnewline
4\A_1 & 1 & \D_8 & [01234567;12345678] \tabularnewline
\A_2 + \A_1 & 2 & \E_8 & 210_x' \tabularnewline
\A_2 + 2\A_1 & 1 & \E_8 & 560_z' \tabularnewline
\A_3 & 1 & \E_8 & 567_x' \tabularnewline
\A_2 + 3\A_1 & 1 & \D_8 & [0123457;1234567] \tabularnewline
2\A_2 & 2 & \E_8 & 700_x' \tabularnewline
2\A_2 + \A_1 & 1 & \E_6\A_2 & 1_p'\boxtimes[123] \tabularnewline
\A_3 + \A_1 & 1 & \E_7\A_1 & 27_a'\boxtimes[12] \tabularnewline
\D_4(a_1) & 6 & \E_8 & 1400_z' \tabularnewline
\D_4 & 1 & \E_8 & 525_x' \tabularnewline
2\A_2 + 2\A_1 & 1 & \A_8 & [123456789] \tabularnewline
\A_3 + 2\A_1 & 1 & \D_8 & [012356;123456] \tabularnewline
\D_4(a_1) + \A_1 & 6 & \E_8 & 1400_x' \tabularnewline
\A_3+\A_2 & 2 & \E_7\A_1 & 56_a\boxtimes[12] \tabularnewline
\A_4 & 2 & \E_8 & 2268_x' \tabularnewline
\A_3+\A_2+\A_1 & 1 & \A_7\A_1 & [12345678]\boxtimes[12] \tabularnewline
\D_4+\A_1 & 1 & \A_7\A_1 & [12345678]\boxtimes[2] \tabularnewline
\D_4(a_1)+\A_2 & 2 & \E_8 & 2240_x' \tabularnewline
\A_4+\A_1 & 2 & \E_8 & 4096_x' \tabularnewline
2\A_3 & 1 & \D_5\A_3 & [01234;12345]\boxtimes [1234] \tabularnewline
\D_5(a_1) & 2 & \E_8 & 2800_z' \tabularnewline
\A_4+2\A_1 & 2 & \E_8 & 4200_x' \tabularnewline
\A_4+\A_2 & 1 & \E_8 & 4536_z' \tabularnewline
\A_5 & 1 & \E_7\A_1 & 168_a'\boxtimes[12] \tabularnewline
\D_5(a_1)+\A_1 & 1 & \E_8 & 6075_x' \tabularnewline
\A_4+\A_2+\A_1 & 1 & \E_8 & 2835_x' \tabularnewline
\D_4+\A_2 & 2 & \D_8 & [0234;1235] \tabularnewline
\E_6(a_3) & 2 & \E_8 & 5600_z' \tabularnewline
\D_5 & 1 & \E_8 & 2100_y \tabularnewline
\A_4+\A_3 & 1 & \A_4\A_4 & [12345]\boxtimes[12345] \tabularnewline
\A_5 + \A_1 & 1 & \A_2\A_1\A_5 & [123]\boxtimes[12]\boxtimes[123456] \tabularnewline
\D_5(a_1)+\A_2 & 1 & \D_5\A_3 & [0124;1234]\boxtimes[1234] \tabularnewline
\D_6(a_2) & 2 & \D_8 & [0134;1245] \tabularnewline
\E_6(a_3)+\A_1 & 2 & \E_6\A_2 & 30_p'\boxtimes[123] \tabularnewline
\E_7(a_5) & 6 & \E_7\A_1 & 315_a\boxtimes[12] \tabularnewline
\D_5+\A_1 & 1 & \E_7\A_1 & 105_c'\boxtimes[12] \tabularnewline
\E_8(a_7) & 120 & \E_8 & 4480_y \tabularnewline
\A_6 & 1 & \E_8 & 4200_z \tabularnewline
\D_6(a_1) & 2 & \E_8 & 5600_z \tabularnewline
\A_6+\A_1 & 1 & \E_8 & 2835_x \tabularnewline
\E_7(a_4) & 2 & \E_7\A_1 & 420_a'\boxtimes[12] \tabularnewline
\E_6(a_1) & 2 & \E_8 & 2800_z \tabularnewline
\D_5+\A_2 & 2 & \D_8 & [024;134] \tabularnewline
\D_6 & 1 & \A_1\A_7 & [2]\boxtimes[2345] \tabularnewline
\E_6 & 1 & \E_8 & 525_x \tabularnewline
\D_7(a_2) & 2 & \E_8 & 4200_x \tabularnewline
\A_7 & 1 & \E_7\A_1 & 210_b\boxtimes[12] \tabularnewline
\E_6(a_1)+\A_1 & 2 & \E_8 & 4096_z \tabularnewline
\E_7(a_3) & 2 & \E_8 & 2268_x \tabularnewline
\E_8(b_6) & 6 & \E_6\A_2 & 80_s\boxtimes[123] \tabularnewline
\D_7(a_1) & 2 & \E_7\A_1 & 405_a\boxtimes[12] \tabularnewline
\E_6+\A_1 & 1 & \A_2\A_1\A_5 & [123]\boxtimes[2]\boxtimes[234] \tabularnewline
\E_7(a_2) & 1 & \A_2\A_1\A_5 & [3]\boxtimes[12]\boxtimes[1245] \tabularnewline
\E_8(a_6) & 6 & \E_8 & 1400_x \tabularnewline
\D_7 & 1 & \A_2\A_1\A_5 & [123]\boxtimes[12]\boxtimes[34] \tabularnewline
\E_8(b_5) & 6 & \E_8 & 1400_z \tabularnewline
\E_7(a_1) & 1 & \E_8 & 567_x \tabularnewline
\E_8(a_5) & 2 & \E_8 & 700_x \tabularnewline
\E_8(b_4) & 2 & \E_7\A_1 & 120_a\boxtimes[12] \tabularnewline
\E_7 & 1 & \A_1\A_7 & [2]\boxtimes[45] \tabularnewline
\E_8(a_4) & 2 & \E_8 & 210_x \tabularnewline
\E_8(a_3) & 2 & \E_8 & 112_z \tabularnewline
\E_8(a_2) & 1 & \E_8 & 35_x \tabularnewline
\E_8(a_1) & 1 & \E_8 & 8_z \tabularnewline
\E_8 & 1 & \E_8 & 1_x \tabularnewline
\end{longtable}

\setstretch{0.96}


\renewcommand*{\bibfont}{\small}
\printbibliography

\setstretch{1.2}
\renewcommand{\indexname}{Index of Notation}
\printindex
\end{document}